\documentclass[12pt]{article}
\usepackage{geometry} 
\usepackage[utf8]{inputenc}
\usepackage[T1]{fontenc}
\usepackage[british]{babel}
\usepackage{lmodern}
\usepackage{graphicx}
\usepackage{amsmath}   
\usepackage{amssymb}   
\usepackage{amsfonts}
\usepackage{amsthm}
\usepackage{mathtools}
\usepackage{bm}
\usepackage{tikz}
\usepackage{url}
\usetikzlibrary{matrix}
\usepackage{multicol}
\usepackage{array}
\usepackage[parfill]{parskip}
\usepackage{fancyhdr}
\usepackage{lipsum}
\usepackage{longtable}
\usepackage{enumerate}
\usepackage{subcaption}
\usepackage{pinlabel}
\usepackage{tikz-cd}
\usepackage{comment}

\usepackage{xcolor}

\usepackage[parfill]{parskip}
\makeatletter
\def\thm@space@setup{%
  \thm@preskip=\parskip \thm@postskip=0pt
}
\makeatother

\numberwithin{equation}{section}
\newtheorem{de}{Definition}[section]
\newtheorem{theo}[de]{Theorem}
\newtheorem{pro}[de]{Proposition}

\newtheorem{co}[de]{Corollary}

\theoremstyle{remark}

\theoremstyle{plain}

\newcommand{\mbP}{\mathbb{P}}
\newcommand{\mbR}{\mathbb{R}}
\newcommand{\mbC}{\mathbb{C}}

\newcommand{\cE}{\mathcal{E}}
\newcommand{\cF}{\mathcal{F}}

\newcommand{\si}{\sigma}

\newcommand{\Pti}{{\tilde{P}}}

\def\t{\widetilde}

\newcommand{\toitself}{\mathbin{\raisebox{-0.4ex}{{\scalebox{1.25}{%
    \lefteqn{\scalebox{.75}{$\blacktriangleleft$}}\raisebox{.45ex}{$\supset$}}}}}}

\allowdisplaybreaks

\title{\textbf{Dynamical degrees of birational maps \\ from indices of polynomials \\
with respect to blow-ups II.\\
3D examples}}
\author{Jaume Alonso \and Kangning Wei \and Yuri B.\ Suris}
\date{\small Institut für Mathematik, MA 7-1\\ Technische Universität Berlin, Str.\ des 17.\ Juni, 10623 Berlin, Germany\\
E-mail: alonso@math.tu-berlin.de, suris@math.tu-berllin.de, wei@math.tu-berlin.de}

\begin{document}

\maketitle

\begin{abstract}
The goal of this paper is the exact computation of the degrees $\deg(f^n)$ of the iterates of birational maps $f: \mathbb{P}^N \dashrightarrow \mathbb{P}^N$. In the preceding companion paper, a new method has been proposed based on the use of indices of polynomials associated to the local blow-ups used to resolve contractions of hypersurfaces by $f$, and on the control of the factorization of pull-backs of polynomials. This leads to recurrence relations for the degrees and the indices. We apply this method to several illustrative examples in three dimensions. These examples demonstrate the flexibility of the method which, in particular, does not require the construction of an algebraically stable lift of $f$, unlike the previously known methods based on the Picard group.
\end{abstract}

\section{Introduction}

The computation of the degrees of the iterates of birational maps of $\mbP^N$ is one of the key problems in the study of birational dynamics \cite{Bedford_Kim_2004}, \cite{BellonViallet1999}, \cite{Diller1996}, \cite{FornaessSibony2}, \cite{McMullen_2007}, \cite{silverman_2012}. This paper is a sequel to \cite{Alonso_Suris_Wei_2D}, in which we proposed a novel technique for this computation. While in that paper the general scheme was formulated and applied to several examples for $N=2$, the present paper is devoted to a detailed exposition of several examples in the case $N=3$, illustrating various aspects and phenomena of interest.

Recall \cite{Alonso_Suris_Wei_2D} that the basic phenomenon responsible for the drop of degree by iterating the birational map $f:\mbP^N\dasharrow\mbP^N$ is the \emph{singularity confinement mechanism}: one has $\deg(f^n)<(\deg(f))^n$ for $n>k$ if and only if there exists a \emph{degree lowering hypersurface}, i.e.\ a hypersurface $\{K=0\}\subset \mathcal C(f)$ such that $f^{k+1}(\{K=0\})\subset \mathcal I(f)$ for some $k\in\mathbb N$.  Recall that the critical set ${\mathcal C}(f)$ and the indeterminacy set $\mathcal I(f)$ are defined in terms of $\tilde f$, a lift of $f$ to the space of homogeneous coordinates $\mbC^{N+1}$ whose components are polynomials without a non-trivial common factor. Namely, ${\mathcal C}(f)$ is defined by the equation $\det d\tilde f=0$, while $\mathcal I(f)$ consist of points where $\tilde f=(0,\ldots,0)$ does not represent a point in $\mbP^N$.

For a precise computation of $\deg(f^n)$ we use the following characterization: this quantity coincides with $\deg(P_n)$, where $P_n$ is the proper $n$-th pull-back of a generic linear polynomial $P_0$ under the action of $f$. The pull-back $f^*P_0=P_0\circ \tilde f$ of an arbitrary homogeneous polynomial on $\mbC^{N+1}$ needs not be irreducible anymore; rather, it can acquire certain factors of geometric origin:
$$
f^*P_0=\left(\prod_{i=1}^r K_i^{\nu_i(P_0)}\right) P_1,
$$
where $P_1$ is irreducible (the proper pull-back of $P_0$), and $K_i$ are minimal generating polynomials of irreducible components of the critical set $\mathcal C(f)$. Iterated proper pullbacks $P_n$ are defined by the recurrent relation
\begin{equation}\label{def Pn}
P_{n+1}=\left(\prod_{i=1}^r K_i^{-\nu_i(P_n)}\right)f^*P_n.
\end{equation}
The definition of degree lowering hypersurface is related to the following fact: $\nu_i(P_0)>0$ if and only if $P_0$ vanishes on $f(\{K_i=0\})$. 
Similarly, if $k\ge 0$ is the smallest index for which $\nu_i(P_k)>0$ (so that $\nu_i(P_j)=0$ for $j=0,1,\ldots,k-1$) then $P_0$ vanishes on $f^{k+1}(\{K_i=0\})$, but not of the earlier iterates of $f(\{K_i=0\})$. If this condition holds true for all coordinate monomials $x_0, x_1,\ldots,x_N$, which means $f^{k+1}(\{K_i=0\})\subset{\mathcal I}(f)$, then all components of $\tilde f^{k+1}$ are divisible by $K_i$, and we obtain $\deg(f^{k+1})<(\deg(f))^{k+1}$.

The first contribution of \cite{Alonso_Suris_Wei_2D} is the exact computation of indices $\nu_i(P)$ for an arbitrary homogeneous polynomial $P$. It is  based on the (local) blow-up maps $\phi_i$ used to resolve the contraction of the divisors $\{K_i=0\}$, and on the evaluation of $P\circ\phi_i$, see Definition 4.3 and Theorem 4.5 in \cite{Alonso_Suris_Wei_2D}.

The second contribution of \cite{Alonso_Suris_Wei_2D} is a semi-formalized algorithm for computing degrees of the iterates of a birational map $f:\mathbb P^N\dasharrow \mathbb P^N$, based on the local indices of polynomials. Its rough outline is as follows.
\begin{itemize}
\item \textbf{Step 1:} For all irreducible components $\{K_i=0\}$ of $\mathcal C(f)$, perform local blow-ups $\phi_i$ of $\mbP^N$ so that  the lift of $f$ to the blow-up variety maps the proper transform of $\{K_i=0\}$ birationally to the corresponding exceptional divisor of $\phi_i$.

\item \textbf{Step 2:} Propagate the blow-ups $\phi_i=\phi_i^{(0)}$ along the orbits of $f(\{K_i=0\})$, so that the exceptional divisor of $\phi_i^{(k)}$ is mapped birationally to the exceptional divisor of $\phi_i^{(k+1)}$.

\item \textbf{Step 3:} Derive the system of recurrent relations between the degrees and indices of $P_n$. For instance, the relation for degrees follows directly from  \eqref{def Pn}:
\begin{equation}
\deg(P_{n+1})=\deg(f)\deg(P_n)-\nu_1^{(0)}(P_n)\deg(K_1)-\ldots-\nu_r^{(0)}(P_n)\deg(K_r).
\end{equation}
Likewise, indices $\nu_i^{(k)}(P_{n+1})$ can be expressed as linear combinations of the indices $\nu_i^{(k+1)}(P_n)$ and the quantities $\nu_j^{(0)}(P_n)$, $\deg(P_n)$. 
\end{itemize}

The latter system is useful if it terminates, i.e., if one can compute some $\nu_i^{(k+1)}(P_n)$ in a non-recursive manner. One of the possibilities for this is that some iterate of $f(\{K_i=0\})$ lands in the indeterminacy set ${\mathcal I}(f)$ and gets blown up to a codimension 1 variety. This is related to the concepts of \emph{singularity confinement}. The action of the lift $f_X$ of $f$ to the blow-up variety $X$ is expressed as follows:
$$
\overline{\{K_i=0\}} \; \dasharrow\; E_i^{(0)}\; \dasharrow\;\ldots \; \dasharrow\; E_i^{(k+1)}\; \dasharrow\; \overline{\{L_i=0\}};
$$
here $E_i^{(j)}$ are exceptional divisors of $\phi_i^{(j)}$, while $\{L_i=0\}$ is an irreducible component of $\mathcal C(f^{-1})$.
This possibility will be illustrated in our Examples 1 and 2, see Sections \ref{sect Example 1}, \ref{sect Example 2}.  Another possibility is a stabilization of indices, when $\nu_i^{(k)}(P_{n+1})$ can be expressed through $\nu_i^{(k)}(P_n)$ rather than through $\nu_i^{(k+1)}(P_n)$. We have the following scheme for the action of $f_X$:
$$
\overline{\{K_i=0\}} \; \dasharrow\; E_i^{(0)}\; \dasharrow\;\ldots \; \dasharrow\; E_i^{(k)}\; \toitself.
$$
This possibility is illustrated in Example 3, see Section \ref{sect Example 3}. Let us mention several important features of these examples.

\begin{itemize}
    \item Example 1 is a generalization of the discrete time Euler top introduced in \cite{HirotaKimura2000}, \cite{PPS2011}, \cite{PS2010}. In this case, all components of the critical set $\mathcal C(f)$ are contracted to points, whose images after two iterates hit $\mathcal I(f)$ (``all singularities are confined''). Moreover, no infinitely near singularities are present, so that all encountered blow-ups are standard sigma-processes (point blow-ups) in $\mbP^3$. Our algorithm  produces an algebraically stable lift $f_X$ of the map $f$, actually a pseudo-automorphism. For this simple example, our technique for computing the degrees of iterates of $f$ is equivalent to the method based in the induced action $f_X^*$ on the Picard group ${\rm Pic}(X)$.
    
    \item Example 2 is a birational map in $\mbP^3$ introduced in \cite{Ercolani2021} and representing a 2D non-autonomous difference equation (discrete Painlev\'e I equation). Here, the critical set $\mathcal C(f)$ consists of two irreducible components, only one of them being degree lowering. To regularize the contraction of the latter component, one has to resolve infinitely near singularities by a sequence of blow-ups combining blow-ups centered at points and at curves. Our algorithm produces a lift of $f$ which allows for a finite (terminating) system of recurrent relations for degrees and indices, but is not algebraically stable. Thus, our method differs in this case substantially from the method based on the Picard group. The latter is also applicable, since $f$ admits another lift which is algebraically stable.

    \item Example 3 is a birational map in $\mbP^3$ introduced in \cite{Viallet_2019}, constructed from an integrable 2D map (a QRT map) via a procedure called \emph{inflation}. In general, inflated maps inherit the number of integrals of motion. In this example, the critical set $\mathcal C(f)$ consists of one degree lowering component. It is contracted to a point, whose fourth iterate belongs to $\mathcal I(f)$ but plays the role of a ``singular fixed point''.    A complete regularization of this would require infinitely many blow-ups, but the system of recurrent relations for degrees and indices is still finite (terminating) due to the phenomenon of stabilization of indices mentioned above. 
\end{itemize}

Let us finish this introduction by mentioning features of 3D maps which make their study more complicated than in the 2D case presented in \cite{Alonso_Suris_Wei_2D}. The algorithm for blowing up for the regularization of contractions and the definition of indices of polynomials with respect to local blow-ups are independent of the dimension of the projective space. However, in dimension two, the available knowledge about the dynamics of birational maps is richer. In particular, according to a theorem of Diller and Favre, see \cite{Birkett2022}, \cite{DillerFavre2001},  any birational map of $\mbP^2$ admits an algebraically stable lift via a sequence of blow-ups. Such a statement is not yet available for birational maps of $\mbP^N$ with $N\ge 3$. It can be hoped that a detailed study of examples, like the one presented here, could be instrumental in achieving a progress towards this goal. 

\section{Example 1: Kahan-Hirota-Kimura discretization of the Euler top, and its generalizations}
\label{sect Example 1}

\subsection{Introduction}
Let $\si: \mbP^3 \dashrightarrow \mbP^3$ be the standard inversion
\begin{equation}\label{3D inversion}
\si: \begin{bmatrix}
X_1 \\[0.2cm] X_2 \\[0.2cm] X_3 \\[0.2cm] X_4
\end{bmatrix}
\mapsto 
\begin{bmatrix}
Y_1 \\[0.2cm] Y_2 \\[0.2cm] Y_3 \\[0.2cm] Y_4
\end{bmatrix}=
\begin{bmatrix}
1/X_1 \\[0.2cm] 1/X_2 \\[0.2cm] 1/X_3 \\[0.2cm] 1/X_4
\end{bmatrix}=
\begin{bmatrix}
X_2 X_3 X_4 \\[0.2cm] X_1 X_3 X_4 \\[0.2cm] X_1 X_2 X_4 \\[0.2cm] X_1 X_2 X_3
\end{bmatrix}
\end{equation}

The critical set and the indeterminacy set:
$$
\mathcal{C}(\si)=\bigcup_{i=1}^4 \Pi_i, \qquad \mathcal{I}(\si)=\bigcup_{1\le i<j\le 4}\ell_{ij},
$$
where $\Pi_i=\{X_i=0\}$ are the coordinate planes and $\ell_{ij}=\Pi_i\cap \Pi_j$ are lines. We introduce  also the four points $q_1=[1:0:0:0]$, $q_2=[0:1:0:0]$, $q_3=[0:0:1:0]$, $q_4=[0:0:0:1]$ and notice that $q_i$ is the intersection point of the three planes: $q_i=\Pi_j\cap\Pi_k\cap\Pi_l$, where $\{i,j,k,l\}=\{1,2,3,4\}$, and, dually, $\Pi_i$ is the plane through $q_j,q_k,q_l$, and also that $\ell_{ij}=(q_kq_l)$.

The points $q_i$ play an important role for the geometry of $\sigma$: it blows down each plane $\Pi_i$ to the point $q_i$, and blows up each point $q_i$ to the plane $\Pi_i$. The involution $\si$ itself is dynamically trivial. However, we get nontrivial dynamical systems by considering the following family:
\begin{equation}\label{3D Ex1 f}
    f =  L_1 \circ \sigma \circ L_2^{-1},
\end{equation}
where $L_1$ and $L_2$ are linear projective automorphisms of $\mbP^3$ \cite{Bedford_Kim_2004}.

For such a map, we have 
$$
\mathcal{C}(f)=\bigcup_{i=1}^4  L_2(\Pi_i), \qquad \mathcal{I}(f)=\bigcup_{1\le i<j\le 4} L_2(\ell_{ij}).
$$ 
and similarly, 
$$
\mathcal{C}(f^{-1})=\bigcup_{i=1}^4  L_1(\Pi_i), \qquad \mathcal{I}(f^{-1})=\bigcup_{1\le i<j\le 4} L_1(\ell_{ij}).
$$
We set $A_i:=L_1(q_i)$ and $B_{i}:=L_2(q_i)$, then the map $f$ blows down the critical plane $L_2(\Pi_i)$ to the point $A_i$ and blows up the point $B_i$ to the plane $L_1(\Pi_i)$.

We consider here the dynamics of the maps of the class \eqref{3D Ex1 f} satisfying the following condition:
\begin{equation}\label{3D Ex1 cond points}
    f(A_i)=B_i,\qquad i=1,2,3,4,
\end{equation}
so that the following singularity patterns take place:
\begin{equation}\label{3D Ex1 sing conf pattern}
{\mathcal C}(f) \supset L_2(\Pi_i) \; \rightarrow \; A_i  \; \rightarrow \; B_i \; \rightarrow \;  L_1(\Pi_i)\subset {\mathcal C}(f^{-1}).
\end{equation}
Condition \eqref{3D Ex1 cond points} says:
$$
(L_1\sigma L_2^{-1})L_1q_i=L_2q_i, \quad i=1,2,3,4,
$$
and this is satisfied if
\begin{equation}\label{3D Ex1 cond matrices}
(L_2^{-1}L_1) \sigma (L_2^{-1}L_1) ={\rm id}.
\end{equation}

An example of matrices $L_1,L_2$ satisfying condition \eqref{3D Ex1 cond matrices} is given by \emph{Kahan-Hirota-Kimura discretization of the Euler top} \cite{HirotaKimura2000}, \cite{PPS2011}, \cite{PS2010}. Recall that this is the birational map on $\mbR^3$, $z\mapsto \t z$, given by the implicit equations of motion
\begin{equation}\label{dET implicit}
\left\{\begin{array}{l}
\t z_1-z_1  = \epsilon\alpha_1(z_2\t z_3+\t z_2z_3), \\
\t z_2-z_2  = \epsilon\alpha_2(z_3\t z_1+\t z_3z_1), \\
\t z_3-z_3 = \epsilon\alpha_3(z_1\t z_2+\t z_1z_2).
\end{array}\right.
\end{equation}
Solving this for $\t z_i$, we find the explicit form of the map:
\begin{equation}\label{dET explicit}
\left\{\begin{array}{l}
\t z_1 = \frac{1}{\Delta}\big(z_1+2\epsilon\alpha_1z_2z_3+\epsilon^2z_1(-\alpha_2\alpha_3z_1^2+\alpha_1\alpha_3z_2^2+\alpha_1\alpha_2z_3^2)\big), \\
\t z_2 = \frac{1}{\Delta}\big(z_2+2\epsilon\alpha_2z_1z_3+\epsilon^2z_2(\alpha_2\alpha_3z_1^2-\alpha_1\alpha_3z_2^2+\alpha_1\alpha_2z_3^2)\big), \\
\t z_3 = \frac{1}{\Delta}\big(z_3+2\epsilon\alpha_3z_1z_2+\epsilon^2z_3(\alpha_2\alpha_3z_1^2+\alpha_1\alpha_2z_2^2-\alpha_1\alpha_2z_3^2)\big), 
\end{array}\right.
\end{equation}
where
\[
\Delta= 1-\epsilon^2(\alpha_2\alpha_3z_1^2+\alpha_1\alpha_3z_2^2+\alpha_1\alpha_2z_3^2)-2\epsilon^3\alpha_1\alpha_2\alpha_3z_1z_2z_3.
\]
In homogeneous coordinates $z_i=x_i/x_4$, $i=1,2,3$, we have:
$$
\left\{
\begin{aligned}
\t x_1 & =  x_4^2x_1+2\epsilon\alpha_1x_4x_2x_3+\epsilon^2x_1(-\alpha_2\alpha_3x_1^2+\alpha_1\alpha_3x_2^2+\alpha_1\alpha_2x_3^2), \\
\t x_2 & =  x_4^2x_2+2\epsilon\alpha_2x_4x_1x_3+\epsilon^2x_2(\alpha_2\alpha_3x_1^2-\alpha_1\alpha_3x_2^2+\alpha_1\alpha_2x_3^2), \\
\t x_3 & =  x_4^2x_3+2\epsilon\alpha_3x_4x_1x_2+\epsilon^2x_3(\alpha_2\alpha_3x_1^2+\alpha_1\alpha_2x_2^2-\alpha_1\alpha_2x_3^2), \\
\t x_4 & =  x_4^3-\epsilon^2x_4(\alpha_2\alpha_3x_1^2+\alpha_1\alpha_3x_2^2+\alpha_1\alpha_2x_3^2)-2\epsilon^3\alpha_1\alpha_2\alpha_3x_1x_2x_3.
\end{aligned} \right.
$$

As observed in \cite{ASW_3D_QRT}, this map belongs to the class \eqref{3D Ex1 f} with
\begin{equation}\label{L1 dET}
L_1^{-1}=
\begin{pmatrix}
b_1 & -b_2 & -b_3 & 1 \\
-b_1 & b_2 & -b_3 & 1 \\
-b_1 & -b_2 & b_3 & 1 \\
b_1 & b_2 & b_3 & 1
\end{pmatrix}=\begin{pmatrix}
1 & -1 & -1 & 1 \\
-1 & 1 & -1 & 1 \\
-1 & -1 & 1 & 1 \\
1 & 1 & 1 & 1
\end{pmatrix}{\rm diag}(b_1,b_2,b_3,1),  
\end{equation}
\begin{equation}\label{L2 dET}
 L_2^{-1} = \begin{pmatrix}
-b_1 & b_2 & b_3 & 1 \\
b_1 & -b_2 & b_3 & 1 \\
b_1 & b_2 & -b_3 & 1 \\
-b_1 & -b_2 & -b_3 & 1
\end{pmatrix}= \begin{pmatrix}
-1 & 1 & 1 & 1 \\
1 & -1 & 1 & 1 \\
1 & 1 & -1 & 1 \\
-1 & -1 & -1 & 1
\end{pmatrix}{\rm diag}(b_1,b_2,b_3,1),
\end{equation}
where $b_i=\epsilon\sqrt{\alpha_j\alpha_k}$. Remarkably, the matrix 
$$
L=L_2^{-1}L_1=\begin{pmatrix}
-1 & 1 & 1 & 1 \\
1 & -1 & 1 & 1 \\
1 & 1 & -1 & 1 \\
-1 & -1 & -1 & 1
\end{pmatrix}
\begin{pmatrix}
1 & -1 & -1 & 1 \\
-1 & 1 & -1 & 1 \\
-1 & -1 & 1 & 1 \\
1 & 1 & 1 & 1
\end{pmatrix}\simeq \begin{pmatrix}
-1 & 1 & 1 & 1 \\
1 & -1 & 1 & 1 \\
1 & 1 & -1 & 1 \\
1 & 1 & 1 & -1
\end{pmatrix}
$$
(where $\simeq$ stands for the projective equivalence) does not depend on the parameters $b_i$. This matrix $L$ is a solution of the equation \eqref{3D Ex1 cond matrices}, that is, $L\circ\sigma\circ L={\rm id}$. A more general solution is:
\begin{equation}\label{3D Ex1 L}
L=L(\gamma)=\begin{pmatrix} -1 & \gamma & 1 & \gamma \\
                              \gamma & -1 & \gamma & 1 \\
                               1 & \gamma & -1 & \gamma \\
                               \gamma & 1 & \gamma & -1
                              \end{pmatrix},
\end{equation} where $\gamma \neq 0$.

\subsection{Regularization of the map}
Actually, upon a conjugation via a projective automorphism $L_2$ of $\mbP^3$, it suffices to study the map 
\begin{equation}\label{f=l sigma}
f=L\circ\sigma,
\end{equation} 
where $L=L_2^{-1}L_1$. Thus, we will assume from now on that $B_i=q_i$, while $A_i=Lq_i$.  To regularise the map $f$ we need to blow up the points $A_i$ and $B_i$, $i=1, \ldots, 4$. The corresponding exceptional divisors will be denoted by $\cE_{i}$ and $\cE_{i+4}$, $i=1, \ldots, 4$. We denote the blow-up variety by $X$ and the lifted map by $f_X$. The patterns \eqref{3D Ex1 sing conf pattern} translate into the following patterns of $f_X$:
\begin{equation}\label{3D Ex1 sing conf pattern-divisor}
    \overline{\Pi}_i\quad \rightarrow \quad \cE_{i}  \quad \rightarrow \quad \cE_{i+4} \quad \rightarrow \quad \overline{L(\Pi_i)} ,
\end{equation}
where, e.g., $\overline{\Pi}_i$ are the proper transforms of the planes $\Pi_i$ in $X$, and the arrows are understood as birational maps. 
On the blow-up variety $X$, $f_X$ is a pseudo-automorphism, in other words, both $f_X$ and $f_X^{-1}$ do not blow down any codimension 1 varieties.  

\subsection{Degree growth}

Recall the main definitions and results of \cite{Alonso_Suris_Wei_2D}. Definition 4.3 there in the present case reads:
\begin{de}
Let $P$ be a homogeneous polynomial in $\mathbb C^4$. Define $\phi_i$, $\phi_{i+4}$ as the blow-ups at the points $A_i=Lq_i$, $B_i=q_i$, $i=1,\ldots,4$. For any $i=1,\ldots,8$ we define the index $\nu_{i}(P) \in {\mathbb N} \cup \{0\}$ by the relation
\begin{equation}
P\circ\phi_i= u_{i}^{\nu_i(P)} \Pti_i(u_i,v_i,w_i),
\end{equation} 
where $\Pti_{i}$ is a polynomial not identically vanishing for $u_{i}=0$. 
\end{de}
Theorem 4.5 of \cite{Alonso_Suris_Wei_2D} in the present case says that, for any homogeneous polynomial on $\mathbb C^4$, its pull-back under $f$ is given by
\begin{equation}
f^*P=\prod_{i=1}^4 K_i^{\nu_i(P)}\tilde P,
\end{equation}
where $K_i(x_0,x_1,x_2,x_3)$ are linear generating polynomials of the planes $\Pi_i$, and the polynomial $\widetilde{P}$ (\emph{the proper pull-back of $P$}) is not divisible by either of the $K_i$. 

\begin{theo}
For any homogeneous polynomial $P$ on $\mathbb C^4$, define recurrently $P_0=P$, and
\begin{equation}\label{defPn}
f^* P_n = \prod_{j=1}^4 K_j^{\nu_j(P_n)} P_{n+1},
\end{equation} 
Then the following recurrent relations hold true:
\begin{eqnarray}
\deg(P_{n+1}) & = & 3 \deg(P_{n}) - \sum_{j=1}^4\nu_j(P_n), \label{3D Ex1 d n+1}\\ 
\nu_{i}(P_{n+1}) & =  & \nu_{i+4}(P_{n}), \label{3D Ex1 nu1234 n+1}\\ 
\nu_{i+4}(P_{n+1}) & = & 2 \deg(P_{n}) - \sum_{j \in\{1,2,3,4\}\setminus\{i\} } \nu_{j}(P_{n}).\label{3D Ex1 nu 5678 n+1}
\end{eqnarray}
(In the last two equations, $i=1,\ldots,4$.) For a generic linear $P_0$, with $\deg(P_0)=1$ and all $\nu_i(P_0)=0$, we have:
\begin{equation}
\deg(P_n)=2 n^2+1.
\end{equation}
\end{theo}
\begin{proof}
From \eqref{defPn} we derive, first, equation \eqref{3D Ex1 d n+1}, and second,
\begin{equation}\label{3D Ex1 nu n+1}
\nu_i(P_{n+1})=\nu_i(f^*P_n)-\sum_{j=1}^4 \nu_j(P_n)\nu_i(K_j), \quad i=1,\ldots,8.
\end{equation}

For $i=1,\ldots,4$ and for any homogeneous polynomial $P$, there holds:
\begin{equation} \label{3D Ex1 nu 1234}
\nu_{i}(f^* P) =  \nu_{i+4}(P).
\end{equation}
Indeed, $f$ acts biregularly in the neighborhood of $A_i$, therefore $\tilde f\circ\phi_i$ is just another parametrization of the blow-up coordinate patch in the neighborhood of $B_i$. Now \eqref{3D Ex1 nu1234 n+1} follows, because $K_j(A_i) \neq 0$ and therefore $\nu_i(K_j)=0$ for all $i,j=1,\ldots,4$.

For $i=1,\ldots,4$ and for any homogeneous polynomial $P$, there holds:
\begin{equation} \label{3D Ex1 nu 5678}
\nu_{i+4}(f^* P) =  2 \deg (P).
\end{equation}
We will give the details for $i=1$.  Recall that 
$$
\phi_5: \; (u_5,v_5,w_5)\mapsto  \;[1:u_5:u_5v_5:u_5w_5]. 
$$
Therefore: 
$$ 
(f^*P) \circ \phi_5=P(\tilde{f}\circ\phi_5)=P\left(L\begin{bmatrix} u_5^3v_5w_5 \\ u_5^2v_5w_5 \\  u_5^2w_5 \\ u_5^2v_5 \end{bmatrix}\right)
=u_5^{2\deg(P)}P\left(L\begin{bmatrix} u_5v_5w_5 \\ v_5w_5 \\  w_5 \\ v_5 \end{bmatrix}\right)
$$ 
There follows $\nu_{5}(f^*P) = 2\deg(P)$. It remains to observe that for $i,j=1,\ldots,4$ we have $K_j(B_i) =0$ if and only if $i \neq j$, and therefore $\nu_{i+4}(K_j) = 1- \delta_{ij}$, where $\delta_{ij}$ is the Kronecker delta. 
\end{proof}
\begin{co}
Degrees of the polynomials $P_n$ satisfy the following recurrence relation:
\begin{equation} \label{3D Ex1 d recur}
\deg(P_{n+1}) = 3\deg(P_n) - 3\deg(P_{n-1}) + \deg(P_{n-2}).
\end{equation}
\end{co}
\begin{proof}
\begin{align*}
\deg(P_{n+1}) &= 3\deg(P_n) - \sum_{i=1}^4 \nu_{i}(P_n) = 3\deg(P_n) - \sum_{i=1}^4 \nu_{i+4}(P_{n-1}) \\
&= 3\deg(P_n) - 8\deg(P_{n-2}) + 3\sum_{i=1}^4 \nu_{i}(P_{n-2}) \\
&= 3\deg(P_n) - 8\deg(P_{n-2}) + 3(3 \deg(P_{n-2}) - \deg(P_{n-1}))\\
&= 3\deg(P_n) - 3\deg(P_{n-1}) + \deg(P_{n-2}).
\end{align*}
\end{proof}

\subsection{Invariant foliations and integrals of motion}

It is interesting that the degree computation is independent of the value of the parameter $\gamma$, even if the dynamics for $\gamma=1$ and for $\gamma \neq 1$ are different. Recall that we consider the map on $\mbP^3$ given by \eqref{f=l sigma}, where $\sigma:\mbP^3\dasharrow \mbP^3$ is the standard inversion given in \eqref{3D inversion}, and the matrix $L=L(\gamma)$ is given in \eqref{3D Ex1 L}.
In coordinates:
\begin{equation}\label{f coord}
f:\left[\begin{array}{c} X_1 \\X_2\\ X_3\\ X_4\end{array}\right]\mapsto \left[\begin{array}{c} \widetilde{X}_1 \\ \widetilde{X}_2 \\ \widetilde{X}_3\\ \widetilde{X}_4\end{array}\right]=
\left[\begin{array}{c} X_2X_4(X_1-X_3)+\gamma X_1X_3(X_2+X_4) \\
                                X_1X_3(X_2-X_4)+\gamma X_2X_4(X_1+X_3) \\ 
                                X_2X_4(X_3-X_1)+\gamma X_1X_3(X_2+X_4) \\ 
                                X_1X_3(X_4-X_2)+\gamma X_2X_4(X_1+X_3) \end{array}\right].
\end{equation}

By considering the divisor classes, the singularity patterns \eqref{3D Ex1 sing conf pattern-divisor} become
\begin{equation}\label{3D Ex1 sing conf pattern-divisor class}
    H-\sum_{j \in\{1,2,3,4\}\setminus\{i\}} \cE_{j+4} \quad \rightarrow \quad \cE_{i}  \quad \rightarrow \quad \cE_{i+4} \quad \rightarrow \quad H-\sum_{j \in\{1,2,3,4\}\setminus\{i\}} \cE_{j} ,
\end{equation}
Summing up all these patterns, one easily finds that the divisor class
$$
K:=2H-\sum_{i=1}^{8} \cE_i
$$
is invariant under the action of $f_X$. 

The linear system $|K|$ consists of all divisors of the divisor class $K$, which corresponds to quadric surfaces passing through the 8 points $A_i$ and $B_i$, $i=1,2,3,4$. It turns out that the dimension of $|K|$ depends on $\gamma$. For a generic $\gamma\neq 1$, the family of quadrics through the eight points is one-dimensional (a pencil):
$$
\mathcal Q_\lambda=\big\{X\in\mbP^3: Q_0(X)-\lambda Q_1(X)=0\big\},
$$
where
\begin{equation}\label{3D Ex1 Q0 Q1}
Q_0(X)=(X_1+X_3)(X_2+X_4), \qquad  Q_1(X)=(X_1-X_3)(X_2-X_4).
\end{equation}
We find:
\begin{eqnarray}
 \widetilde{X}_1-\widetilde{X}_3 & = & 2X_2X_4(X_1-X_3),\label{planes1} \\
 \widetilde{X}_2-\widetilde{X}_4 & = & 2X_1X_3(X_2-X_4), \label{planes2}\\
 \widetilde{X}_1+\widetilde{X}_3 & = & 2\gamma X_1X_3(X_2+X_4), \label{planes3}\\
 \widetilde{X}_2+\widetilde{X}_4 & = & 2\gamma X_2X_4(X_1+X_3). \label{planes4}
\end{eqnarray}
There follows that the foliation of $\mbP^3$ by quadrics $\mathcal Q_{\lambda}$ is invariant under $f$, more precisely $f$ maps $\mathcal Q_{\lambda}$ to $\mathcal Q_{\gamma^{2}\lambda}$. It is instructive to compute the restriction of $f$ to $\mathcal Q_{\lambda}$. For this, one can parametrize each $\mathcal Q_{\lambda}$ by $(x,y)\in\mbP\times\mbP$ according to
  \begin{equation}
 \left[\begin{array}{c} X_1 \\ X_2\\ X_3\\ X_4\end{array}\right]=
 \left[\begin{array}{c} x+\lambda^{-1}xy \\
                                 y+1 \\ 
                                  x-\lambda^{-1}xy\\ 
                                  y-1\end{array}\right].
 \end{equation}
 The inverse map $\mathcal Q_{\lambda}\to \mbP\times \mbP$ is given by 
 \begin{equation}
 x=\frac{X_1+X_3}{X_2-X_4}=\frac{\lambda(X_1-X_3)}{X_2+X_4}, \qquad y=\frac{X_2+X_4}{X_2-X_4}=\frac{\lambda(X_1-X_3)}{X_1+X_3}.
 \end{equation}
 We already know that $f$ maps $\lambda\mapsto \widetilde{\lambda}=\gamma^{2} \lambda$. We can now use \eqref{planes1}--\eqref{planes4} to compute
 $$
 \widetilde{x}=\frac{\widetilde{X}_1+\widetilde{X}_3}{\widetilde{X}_2-\widetilde{X}_4}=\frac{2\gamma X_1X_3(X_2+X_4)}{2X_1X_3(X_2-X_4)}=\frac{\gamma(X_2+X_4)}{X_2-X_4}=\gamma y,
 $$
 $$
 \widetilde{y}=\frac{\widetilde{X}_2+\widetilde{X}_4}{\widetilde{X}_2-\widetilde{X}_4}=\frac{2\gamma X_2X_4(X_1+X_3)}{2X_1X_3(X_2-X_4)}=\gamma x\,\frac{y^2-1}{x^2-\lambda^{-2}x^2y^2}=\frac{\gamma}{x}\,\frac{y^2-1}{1-\lambda^{-2}y^2}.
 $$
This result can be represented as a non-autonomous difference equation of the second order: 
 $$
 y_{n+1}y_{n-1}=\frac{y_n^2-1}{1-\lambda_n^{-2}y_n^2}, \quad \text{where}\quad \lambda_n=\lambda_0\gamma^{2n}.
 $$
This equation is known in the literature as ``discrete Painlev\'e III equation''. 
 \smallskip
 
 The situation changes dramatically if $\gamma=1$. In this case, the eight points $A_i$, $B_i$ are in a special configuration (associated points), so that there exists a two-dimensional family of quadrics (a net) through them:
 $$
 \mathcal Q_{\lambda_1,\lambda_2}=\big\{X\in\mbP^3: Q_0(X)-\lambda_1 Q_1(X)-\lambda_2 Q_2(X)=0\big\},
 $$
 where 
 $$
 Q_2(X)=X_1X_3+X_2X_4.
 $$ 
 The map $f$ admits two integrals of motion, $J_1(X)=Q_0(X)/Q_1(X)$ and $J_2(X)=Q_2(X)/Q_0(X)$. The first of them means that $f$ maps each $\mathcal Q_\lambda$ to itself, and on each $\mathcal Q_\lambda$, $f$ acts as a QRT map (or, more precisely, as a QRT-root)
    \begin{equation}
 (\widetilde{x},\widetilde{y})=\Big(y, \frac{1}{x}\, \frac{y^2-1}{1-\lambda^{-2}y^2}\Big),
 \end{equation}
 depending on $\lambda$ as a parameter. It has an integral of motion 
 \begin{equation}
 J_2(x,y)=\Big(\frac{x}{y}+\frac{y}{x}\Big)-\Big(\frac{1}{xy}+\lambda^{-2}xy\Big).
 \end{equation}


\section{Example 2: A 3D form of the discrete Painlev\'e I equation after Ercolani et al.}
\label{sect Example 2}

\subsection{Introduction}

We borrow the following example from \cite{Ercolani2021}, where the non-autonomous discrete Painlev\'e I equation (dPI) is transformed into an equivalent 3D autonomous system. They consider the following form of dPI:
\begin{equation}\label{dP1}
u_{n+1}+u_n+u_{n-1}=\frac{\gamma n+c_0}{u_n}+1,
\end{equation}
assuming the genericity condition $\gamma\neq 0$.
Setting $c_n=\gamma n+c_0$ and $v_n=u_{n-1}$, this takes the form of an autonomous birational 3D map (here we drop indices and replace the shift of indices $n\to n+1$ by the tilde symbol):
\begin{equation}\label{dP1 3D}
\begin{pmatrix} u \\ v \\ c\end{pmatrix}\mapsto \begin{pmatrix} \t u \\ \t v \\ \t c\end{pmatrix}=
\begin{pmatrix}-u-v+\dfrac{c}{u}+1 \\ u \\ c+\gamma\end{pmatrix}.
\end{equation}
Next, one performs a birational change of variables (we do not discuss its origin here):
\begin{equation}\label{dP1 change}
x=\frac{u+v-1}{u}, \quad y=\frac{1}{u}\Big(\frac{c}{u}-v\Big), \quad z= \frac{1}{u},
\end{equation}
so that
\begin{equation}\label{dP1 change inv}
u=\frac{1}{z}, \quad v=\frac{x+z-1}{z}, \quad c=\frac{x+y+z-1}{z^2}.
\end{equation}
Then one easily computes map \eqref{dP1 3D} in new variables:
\begin{equation}\label{dP1 3D Erc}
\begin{pmatrix} x \\ y \\ z\end{pmatrix}\mapsto \begin{pmatrix} \t x \\ \t y \\ \t z\end{pmatrix}=
\begin{pmatrix} \dfrac{y}{y+z-1} \\ \dfrac{x+\gamma z^2}{(y+z-1)^2}  \\ \dfrac{z}{y+z-1}\end{pmatrix}.
\end{equation}
Thus, in homogeneous coordinates we come to the birational map $f:\mbP^3\dasharrow\mbP^3$ of degree 2, 
\begin{equation}\label{dP1 f}
 f: \left[\begin{array}{c} x\\y\\z\\w\end{array}\right]\mapsto \left[\begin{array}{c} y(y+z-w) \\ xw+\gamma z^2 \\ z(y+z-w)\\(y+z-w)^2\end{array}\right],
\end{equation}
which we study in the current section. The inverse map is given by
\begin{equation}\label{dP1 f inv}
 f^{-1}: \left[\begin{array}{c} x\\y\\z\\w\end{array}\right]\mapsto\left[\begin{array}{c} yw-\gamma z^2 \\ x(x+z-w) \\ z(x+z-w)\\(x+z-w)^2\end{array}\right].
\end{equation}

The indeterminacy set of $f$ is the plane conic:
\begin{equation}\label{I(f)}
\mathcal I(f)=\{y+z-w=0,\;xw+\gamma z^2=0\},
\end{equation}
and similarly for the inverse map:
\begin{equation}\label{I(f inv)}
\mathcal I(f^{-1})=\{x+z-w=0,\; yw-\gamma z^2=0\}.
\end{equation}
The critical sets of the map $f$ and $f^{-1}$ are pairs of planes:
\begin{equation}\label{C(f)}
\mathcal C(f)=\{w=0\}\cup \{y+z-w=0\},
\end{equation}
and 
\begin{equation}\label{C(f inv)}
\mathcal C(f^{-1})=\{w=0\}\cup \{x+z-w=0\}.
\end{equation}

The fate of the two critical components of $\mathcal C(f)$ is different. The plane $\{w=0\}$ is not degree lowering: it is blown down to the plane curve $\mathcal I(f^{-1})$, which becomes under further iterations an increasingly complicated curve, not being part of $\mathcal I(f)$. The plane $\{y+z-w=0\}$, on the other hand, is degree lowering:
\begin{equation}
f: \quad \{y+z-w=0\} \;\rightarrow \; p_1 \;\rightarrow \; p_2\;\rightarrow \; p_3 \in \mathcal I(f),
\end{equation}
where
\begin{equation}
p_1=[0:1:0:0], \quad p_2=[1:0:0:1], \quad p_3=[0:1:0:1].
\end{equation}

We will regularize this map following the prescription from \cite{Alonso_Suris_Wei_2D}, via blowing up the orbit of $p_1$, and then we will compute the degrees of $f^n$ with the help of the local indices of polynomials at the relevant blow-ups.  

\subsection{Regularization of the map $f$ by blow-ups}
\label{sect Example2 blow-up 1}

\paragraph{Action of $f$ on the critical plane $\{y+z-w=0\}$.} Consider the fate of the critical plane $\{y+z-w=0\}$ under $f$. It is contracted to the point $p_1=[0:1:0:0]$. We can resolve this with two steps of blowing ups:

\begin{itemize}
    \item At the first step, we blow up the point $p_1$:
\begin{equation}\label{dP1 pi1}
\pi_1: \; \left\{\begin{aligned}
x&=u_1,\\y&=1,\\z&=u_1v_1,\\w&=u_1w_1.\end{aligned}\right.
\end{equation}
The exceptional divisor of $\pi_1$ is given by $E_1=\{u_1=0\}$. In the coordinates \eqref{dP1 pi1} for the image $\t x:\t y:\t z:\t w$, we compute:
$$
\t v_1=\frac{\t z}{\t x}=\frac{z}{y}, \quad \t w_1=\frac{\t w}{\t x}=\frac{y+z-w}{y}.
$$
It follows that the image of (the proper transform of) the plane $\{y+z-w=0\}$ in the exceptional divisor is the curve  $\{w_1=0\}\subset E_1$.

   \item At the second step, we blow up this curve $\{w_1=0\}$ in $E_1$:
\begin{equation}\label{dP1 pi2}
\pi_2:\; \left\{\begin{aligned} u_1&=u_2, \\ v_1&=v_2, \\w_1&=w_2u_2,  \end{aligned}\right. \quad \Rightarrow\quad 
\phi_2=\pi_1\circ \pi_2: \;\left\{\begin{aligned} x&=u_2, \\ y&=1,\\ z&=u_2v_2, \\ w&=u_2^2w_2. \end{aligned}\right. 
\end{equation}
The exceptional divisor of $\pi_2$ is given by $E_2=\{u_2=0\}$. In the coordinates \eqref{dP1 pi2} for the image $\t x:\t y:\t z:\t w$, we compute:
 $$
\t v_2=\frac{\t z}{\t x}=\frac{z}{y}, \quad \t w_2=\frac{\t w\t y}{\t x^2}=\frac{xw+\gamma z^2}{y^2}.
$$  
Thus, (the proper transform of) the plane $\{y+z-w=0\}$ is mapped birationally to the exceptional divisor $E_2$:
$$
f|_{\{y+z-w=0\}}: [x:y:z:y+z]\mapsto (0,\t v_2,\t w_2)=\big(0,\frac{z}{y},\frac{x(y+z)+\gamma z^2}{y^2}\big).
$$
\end{itemize}

\paragraph{Action of $f$ on $E_2$.}
Next, we consider the fate of the exceptional divisor $E_2$ under (the lifts of) $f$. The lift of $f$ by $\phi_2$ contracts $E_2$ to the point $p_2=[1:0:0:1]$.
To resolve this, we will have to perform a sequence of three blow-ups at $p_2$.

\begin{itemize}

\item At the first step, we blow up the point $p_2$: 
\begin{equation}\label{dP1 pi3}
  \pi_3:\;  \left\{\begin{aligned} x&=1,\\ y&=u_3v_3, \\ z&=u_3, \\ w&=1+u_3w_3.\end{aligned}\right.
\end{equation}
The exceptional divisor of $\pi_3$ is given by $E_3=\{u_3=0\}$. A straightforward computation shows that the lift of $f$ by $\pi_3$ sends $E_2$ to the point $(u_3,v_3,w_3)=(0,0,1)$ in $E_3$.

\item At the second step, we blow up the point $(u_3,v_3,w_3)=(0,0,1)$ in $E_3$:
\begin{equation}\label{dP1 pi4}
  \pi_4:\;     \left\{\begin{aligned} u_3&=u_4,\\ v_3&=u_4v_4, \\ w_3&=1+u_4w_4,  \end{aligned}\right. \quad \Rightarrow \quad 
  \pi_3\circ \pi_4:\;     \left\{\begin{aligned} x&=1,\\ y&=u_4^2v_4, \\ z&=u_4, \\ w&=1+u_4+u_4^2w_4.\end{aligned}\right.
\end{equation}
The exceptional divisor of $\pi_4$ is given by $E_4=\{u_4=0\}$. We compute that the lift of $f$ by $\pi_3\circ \pi_4$ sends $E_2$ to the curve $\{v_4=\gamma\}$ in $E_4$.
    
\item At the third step, we blow up the curve $\{v_4=\gamma\}$ in $E_4$:
\begin{equation}\label{dP1 pi5}
 \pi_5:\;  \left\{\begin{aligned} u_4&=u_5,\\ v_4&=\gamma+u_5v_5, \\ w_4&=w_5,  \end{aligned}\right. \quad \Rightarrow \quad 
 \phi_5=\pi_3\circ\pi_4\circ\pi_5:\;    \left\{\begin{aligned} x&=1,\\ y&=\gamma u_5^2+u_5^3v_5, \\ z&=u_5, \\ w&=1+u_5+u_5^2w_5.\end{aligned}\right.
\end{equation}
The exceptional divisor of $\pi_5$ is given by $E_5=\{u_5=0\}$. To determine the image of the exceptional divisor $E_2$ under the lift of $f$ by $\phi_5$, we compute:
$$
  \t v_5=\left.\frac{\t x(\t x\widetilde{y}-\gamma \widetilde{z}^2)}{\widetilde{z}^3}\right|_{u_2=0}=\frac{w_2}{v_2^3}-\gamma, \quad 
  \t w_5=\left.\frac{\t x(\widetilde{w}-\t x-\widetilde{z})}{\widetilde{z}^2}\right|_{u_2=0}=-\frac{w_2}{v_2^2},
$$
which gives a birational map of $E_2$ to $E_5$. 

\end{itemize}

\paragraph{Action of $f$ on $E_5$.} Next, we consider the fate of the exceptional divisor $E_5$ under (the lifts of) $f$. The lift of $f$ by $\phi_5$ contracts $E_5$ to the point $p_3=[0:1:0:1]$. To resolve this, we will have to perform a sequence of three blow-ups at $p_3$.

\begin{itemize} 

\item At the first step, we blow up the point $p_3$:
\begin{equation}\label{dP1 pi6}
  \pi_6:\; \left\{\begin{aligned} x&=u_6v_6,\\ y&=1+u_6w_6, \\z&=u_6, \\ w&=1.\end{aligned}\right.
\end{equation}
The exceptional divisor of $\pi_6$ is given by $E_6=\{u_6=0\}$. The lift of $f$ by $\pi_6$ sends $E_5$ to the point $(u_6,v_6,w_6)=(0,0,-1)$ in $E_6$.
     
\item At the second step, we blow up the point  $(u_6,v_6,w_6)=(0,0,-1)$ in $E_6$:
\begin{equation}\label{dP1 pi7}
   \pi_7:\; \left\{\begin{aligned} u_6&=u_7,\\ v_6&=u_7v_7, \\ w_6&=-1+u_7w_7,  \end{aligned}\right. \quad \Rightarrow \quad 
   \pi_6\circ\pi_7:\;  \left\{\begin{aligned} x&=u_7^2v_7,\\ y&=1-u_7+u_7^2w_7, \\ z&=u_7, \\ w&=1. \end{aligned}\right.
\end{equation}
The exceptional divisor of $\pi_7$ is given by $E_7=\{u_7=0\}$.  The lift of $f$ by $\pi_6\circ\pi_7$ sends $E_5$ to the curve $\{v_7=-\gamma\}$ in $E_7$.

\item At the third step, we blow up the curve $\{v_7=-\gamma\}$ in $E_7$:
\begin{equation}\label{dP1 pi8}
  \pi_8:\;  \left\{\begin{aligned} u_7&=u_8, \\ v_7&=-\gamma+v_8u_8,\\ w_7&=w_8,  \end{aligned}\right. \quad \Rightarrow \quad 
  \phi_8=\pi_6\circ\pi_7\circ\pi_8:\;  \left\{\begin{aligned} x&=-\gamma u_8^2+u_8^3v_8,\\ y&=1-u_8+u_8^2w_8, \\ z&=u_8, \\ w&=1. \end{aligned}\right.
\end{equation}
The exceptional divisor of $\pi_8$ is given by $E_8=\{u_8=0\}$. To determine the image of $E_5=\{u_5=0\}$ in $E_8$ under the lift of $f$ by $\phi_8$, we compute:
$$
    \t v_8=\left.\frac{(\widetilde{x}\t w+\gamma \widetilde{z}^2)\t w}{\widetilde{z}^3}\right|_{u_5=0}=v_5, \quad 
    \t w_8=\left.\frac{(\widetilde{y}-\t w+\widetilde{z})\t w}{\widetilde{z}^2}\right|_{u_5=0}=-w_5+3\gamma.
$$
This gives a birational map of $E_5$ to $E_8$. 
\end{itemize}

\paragraph{Action of $f$ on $E_8$.} We compute straightforwardly that the lift of $f$ by $\phi_8$ contracts $E_8\setminus \{w_8=0\}$ to the point $p_4=[1:0:0:0]$. Indeed, substituting formulas \eqref{dP1 pi8} for $(x,y,z,w)$ into \eqref{dP1 f}, we find: $[\t x:\t y:\t z:\t w]=[w_8(1-u_8+u_8^2w_8):u_8v_8:u_8w_8:u_8^2w_8^2]$, which turns into $[w_8:0:0:0]$ at $u_8=0$. To resolve this, we will perform a sequence of two blow-ups at $p_4$.

\begin{itemize}
    
\item At the first step, we blow up the point $[1:0:0:0]$:
\begin{equation}\label{dP1 pi9}
 \pi_9:\;   \left\{\begin{aligned} x&=1,\\ y&=u_9v_9,\\ z&=u_9,\\ w&=u_9w_9.\end{aligned}\right.
\end{equation}
The exceptional divisor of $\pi_9$ is given by $E_9=\{u_9=0\}$.  The lift of $f$ by $\pi_9$ sends $E_8$ to the curve $\{w_9=0\}$ in $E_9$.
    
\item At the second step, we blow up the curve $\{w_9=0\}$ in $E_9$:
\begin{equation}\label{dP1 pi10}
 \pi_{10}:\;    \left\{\begin{aligned} u_9&=u_{10}, \\ v_9&=v_{10}, \\ w_9&=w_{10}u_{10}, \end{aligned}\right. \quad \Rightarrow \quad 
 \phi_{10}=\pi_8\circ\pi_9\circ\pi_{10}:\;  \left\{\begin{aligned} x&=1,\\ y&=u_{10}v_{10}, \\ z&=u_{10}, \\ w&=u_{10}^2w_{10}. \end{aligned}\right.    
\end{equation}
The exceptional divisor of $\pi_{10}$ is given by $E_{10}=\{u_{10}=0\}$.   To determine the image of the exceptional divisor $E_8$ in $E_{10}$ under the lift of $f$ by $\phi_{10}$, we compute:
$$
  \t  v_{10}=\left.\frac{\widetilde{y}}{\widetilde{z}}\right|_{u_8=0}=\frac{v_8}{w_8}, \quad 
  \t  w_{10}=\left.\frac{\t x\widetilde{w}}{\widetilde{z}^2}\right|_{u_8=0}=w_8,
    $$
    which gives a birational map of $E_8$ to $E_{10}$. 
\end{itemize}

\paragraph{Action of $f$ on $E_{10}$.}
Near the exceptional divisor $E_{10}$, we have
\begin{equation}\label{ex2 eq B42}
   f :\left[\begin{array}{c}  1\\ u_{10}v_{10} \\  u_{10} \\ u_{10}^2w_{10} \end{array}\right]
   \mapsto \left[\begin{array}{c} u_{10}^2v_{10} (1+v_{10}-u_{10}w_{10})\\ u_{10}^2(w_{10}+\gamma) \\ u_{10}^2 (1+v_{10}-u_{10}w_{10})\\ u_{10}^2(1+v_{10}-u_{10}w_{10})^2 \end{array}\right]
   = \left[\begin{array}{c} v_{10} (1+v_{10}-u_{10}w_{10})\\ w_{10}+\gamma \\ 1+v_{10}-u_{10}w_{10} \\ (1+v_{10}-u_{10}w_{10})^2 \end{array}\right].
\end{equation}
At $u_{10}=0$, we find
$$
\left[\begin{array}{c} v_{10} (1+v_{10}) \\ w_{10}+\gamma \\ 1+v_{10} \\ (1+v_{10})^2 \end{array}\right],
$$
which represents the plane $\{x+z-w=0\}$. Thus, the lift of $f$ by $\phi_{10}$ maps birationally the exceptional divisor $E_{10}$ to (the proper transform of) this plane.

Thus, we regularized the action of $f$ on the critical plane $\{y+z-w=0\}$ by blowing up $\mbP^3$ at the four points $p_1$, $p_2$, $p_3$, $p_4$ via $\phi_2$, $\phi_5$, $\phi_8$ and $\phi_{10}$, respectively. We denote the blow-up variety by $X$ and the lift of $f$ to this variety by $f_X$. The result is summarized on Figure \ref{fig:1}. 
\begin{figure}[h]
        \centering
\tikzset{every picture/.style={line width=0.75pt}} 

\begin{tikzpicture}[x=0.75pt,y=0.75pt,yscale=-1,xscale=1]

\draw   (50,350) -- (100,350) -- (100,400) -- (50,400) -- cycle ;
\draw  [color={rgb, 255:red, 208; green, 2; blue, 27 }  ,draw opacity=1 ][fill={rgb, 255:red, 208; green, 2; blue, 27 }  ,fill opacity=1 ] (273,337.5) .. controls (273,336.12) and (274.12,335) .. (275.5,335) .. controls (276.88,335) and (278,336.12) .. (278,337.5) .. controls (278,338.88) and (276.88,340) .. (275.5,340) .. controls (274.12,340) and (273,338.88) .. (273,337.5) -- cycle ;
\draw   (550,350) -- (600,350) -- (600,400) -- (550,400) -- cycle ;
\draw   (150,300) -- (200,300) -- (200,350) -- (150,350) -- cycle ;
\draw   (250,300) -- (300,300) -- (300,350) -- (250,350) -- cycle ;
\draw   (350,300) -- (400,300) -- (400,350) -- (350,350) -- cycle ;
\draw   (450,300) -- (500,300) -- (500,350) -- (450,350) -- cycle ;
\draw   (150,225) -- (200,225) -- (200,275) -- (150,275) -- cycle ;
\draw   (250,225) -- (300,225) -- (300,275) -- (250,275) -- cycle ;
\draw   (350,225) -- (400,225) -- (400,275) -- (350,275) -- cycle ;
\draw   (450,225) -- (500,225) -- (500,275) -- (450,275) -- cycle ;
\draw   (250,150) -- (300,150) -- (300,200) -- (250,200) -- cycle ;
\draw   (350,150) -- (400,150) -- (400,200) -- (350,200) -- cycle ;
\draw [color={rgb, 255:red, 208; green, 2; blue, 27 }  ,draw opacity=1 ][line width=1.5]    (160,340) -- (190,310) ;
\draw [color={rgb, 255:red, 208; green, 2; blue, 27 }  ,draw opacity=1 ][line width=1.5]    (260,265) -- (290,235) ;
\draw [color={rgb, 255:red, 208; green, 2; blue, 27 }  ,draw opacity=1 ][line width=1.5]    (360,265) -- (390,235) ;
\draw [color={rgb, 255:red, 208; green, 2; blue, 27 }  ,draw opacity=1 ][line width=1.5]    (460,340) -- (490,310) ;
\draw [color={rgb, 255:red, 208; green, 2; blue, 27 }  ,draw opacity=1 ][line width=1.5]    (190,360) ;
\draw  [color={rgb, 255:red, 208; green, 2; blue, 27 }  ,draw opacity=1 ][fill={rgb, 255:red, 208; green, 2; blue, 27 }  ,fill opacity=1 ] (173,387.5) .. controls (173,386.12) and (174.12,385) .. (175.5,385) .. controls (176.88,385) and (178,386.12) .. (178,387.5) .. controls (178,388.88) and (176.88,390) .. (175.5,390) .. controls (174.12,390) and (173,388.88) .. (173,387.5) -- cycle ;
\draw  [dash pattern={on 4.5pt off 4.5pt}]  (75,340) -- (138.74,261.55) ;
\draw [shift={(140,260)}, rotate = 129.09] [color={rgb, 255:red, 0; green, 0; blue, 0 }  ][line width=0.75]    (10.93,-3.29) .. controls (6.95,-1.4) and (3.31,-0.3) .. (0,0) .. controls (3.31,0.3) and (6.95,1.4) .. (10.93,3.29)   ;
\draw  [dash pattern={on 4.5pt off 4.5pt}]  (185,215) -- (238.38,176.18) ;
\draw [shift={(240,175)}, rotate = 143.97] [color={rgb, 255:red, 0; green, 0; blue, 0 }  ][line width=0.75]    (10.93,-3.29) .. controls (6.95,-1.4) and (3.31,-0.3) .. (0,0) .. controls (3.31,0.3) and (6.95,1.4) .. (10.93,3.29)   ;
\draw  [dash pattern={on 4.5pt off 4.5pt}]  (410,175) -- (463.38,213.82) ;
\draw [shift={(465,215)}, rotate = 216.03] [color={rgb, 255:red, 0; green, 0; blue, 0 }  ][line width=0.75]    (10.93,-3.29) .. controls (6.95,-1.4) and (3.31,-0.3) .. (0,0) .. controls (3.31,0.3) and (6.95,1.4) .. (10.93,3.29)   ;
\draw  [dash pattern={on 4.5pt off 4.5pt}]  (310,170) -- (338,170) ;
\draw [shift={(340,170)}, rotate = 180] [color={rgb, 255:red, 0; green, 0; blue, 0 }  ][line width=0.75]    (10.93,-3.29) .. controls (6.95,-1.4) and (3.31,-0.3) .. (0,0) .. controls (3.31,0.3) and (6.95,1.4) .. (10.93,3.29)   ;
\draw  [dash pattern={on 4.5pt off 4.5pt}]  (510,260) -- (573.74,338.45) ;
\draw [shift={(575,340)}, rotate = 230.91] [color={rgb, 255:red, 0; green, 0; blue, 0 }  ][line width=0.75]    (10.93,-3.29) .. controls (6.95,-1.4) and (3.31,-0.3) .. (0,0) .. controls (3.31,0.3) and (6.95,1.4) .. (10.93,3.29)   ;
\draw  [color={rgb, 255:red, 208; green, 2; blue, 27 }  ,draw opacity=1 ][fill={rgb, 255:red, 208; green, 2; blue, 27 }  ,fill opacity=1 ] (273,387.5) .. controls (273,386.12) and (274.12,385) .. (275.5,385) .. controls (276.88,385) and (278,386.12) .. (278,387.5) .. controls (278,388.88) and (276.88,390) .. (275.5,390) .. controls (274.12,390) and (273,388.88) .. (273,387.5) -- cycle ;
\draw  [color={rgb, 255:red, 208; green, 2; blue, 27 }  ,draw opacity=1 ][fill={rgb, 255:red, 208; green, 2; blue, 27 }  ,fill opacity=1 ] (373,387.5) .. controls (373,386.12) and (374.12,385) .. (375.5,385) .. controls (376.88,385) and (378,386.12) .. (378,387.5) .. controls (378,388.88) and (376.88,390) .. (375.5,390) .. controls (374.12,390) and (373,388.88) .. (373,387.5) -- cycle ;
\draw  [color={rgb, 255:red, 208; green, 2; blue, 27 }  ,draw opacity=1 ][fill={rgb, 255:red, 208; green, 2; blue, 27 }  ,fill opacity=1 ] (373,337.5) .. controls (373,336.12) and (374.12,335) .. (375.5,335) .. controls (376.88,335) and (378,336.12) .. (378,337.5) .. controls (378,338.88) and (376.88,340) .. (375.5,340) .. controls (374.12,340) and (373,338.88) .. (373,337.5) -- cycle ;
\draw  [color={rgb, 255:red, 208; green, 2; blue, 27 }  ,draw opacity=1 ][fill={rgb, 255:red, 208; green, 2; blue, 27 }  ,fill opacity=1 ] (473,387.5) .. controls (473,386.12) and (474.12,385) .. (475.5,385) .. controls (476.88,385) and (478,386.12) .. (478,387.5) .. controls (478,388.88) and (476.88,390) .. (475.5,390) .. controls (474.12,390) and (473,388.88) .. (473,387.5) -- cycle ;

\draw (169,403) node [anchor=north west][inner sep=0.75pt]   [align=left] {$\displaystyle p_{1}$};
\draw (269,403) node [anchor=north west][inner sep=0.75pt]   [align=left] {$\displaystyle p_{2}$};
\draw (369,403) node [anchor=north west][inner sep=0.75pt]   [align=left] {$\displaystyle p_{3}$};
\draw (469,403) node [anchor=north west][inner sep=0.75pt]   [align=left] {$\displaystyle p_{4}$};
\draw (27,404) node [anchor=north west][inner sep=0.75pt]    {$\{y+z-w=0\}$};
\draw (527,404) node [anchor=north west][inner sep=0.75pt]    {$\{x+z-w=0\}$};
\draw (153,303) node [anchor=north west][inner sep=0.75pt]    {$E_{1}$};
\draw (153,228) node [anchor=north west][inner sep=0.75pt]    {$E_{2}$};
\draw (254,303) node [anchor=north west][inner sep=0.75pt]    {$E_{3}$};
\draw (253,228) node [anchor=north west][inner sep=0.75pt]    {$E_{4}$};
\draw (253,153) node [anchor=north west][inner sep=0.75pt]    {$E_{5}$};
\draw (353,303) node [anchor=north west][inner sep=0.75pt]    {$E_{6}$};
\draw (353,228) node [anchor=north west][inner sep=0.75pt]    {$E_{7}$};
\draw (353,153) node [anchor=north west][inner sep=0.75pt]    {$E_{8}$};
\draw (452,228) node [anchor=north west][inner sep=0.75pt]    {$E_{10}$};
\draw (453,303) node [anchor=north west][inner sep=0.75pt]    {$E_{9}$};

\end{tikzpicture}
 \caption{Blow-ups for the degree computation via the method of indices of polynomials. Blow-ups maps relate vertically aligned objects. The red color is used for the centers of the corresponding blow-ups (points and curves). For each of the ten blow-ups, the exceptional divisor is placed directly above the corresponding center. Dashed arrows represent birational maps between the top level exceptional divisors. }
        \label{fig:1}

\end{figure}
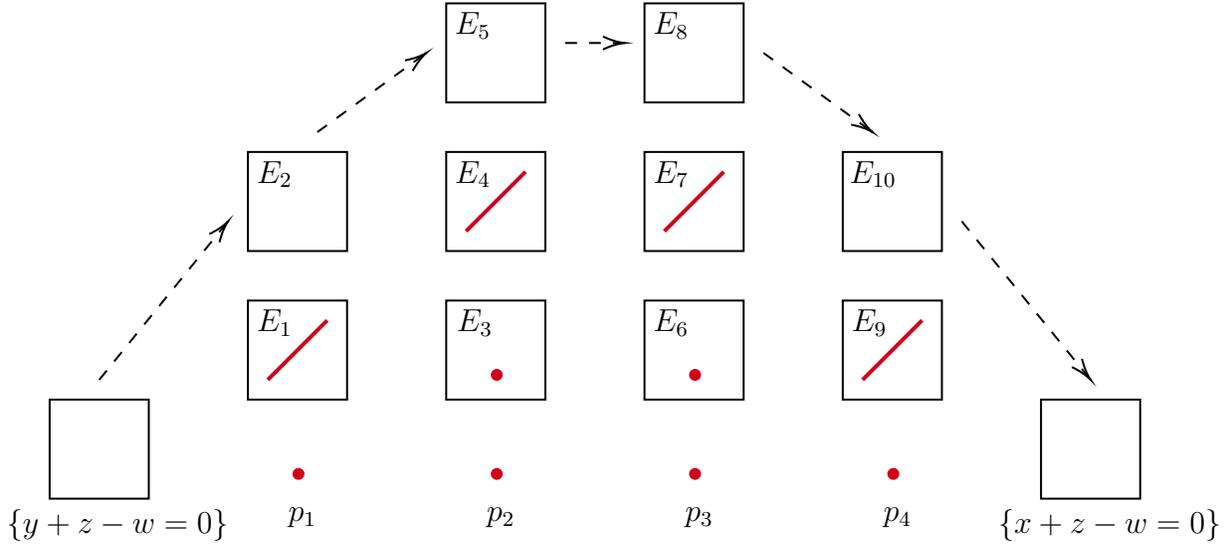

 The lift $f_X$ constructed in this section is sufficient for the degree computations, which will be performed in Section \ref{sect Example 2 degree computation}, but it is not algebraically stable. In fact, $f_X$ blows down the divisor $E_1$ to the curve $\{v_3=0\}\subset E_3$, which is mapped to the curve $\{v_6=0\}\in E_6$, which lies in the singularity set $\mathcal I(f_X)$. Indeed, computing the last three coordinates of $\tilde f\circ \pi_6$, which can be considered as homogeneous coordinates of the image of a neighborhood of $E_6$ in the coordinate patch \eqref{dP1 pi9} around $E_9$, we obtain
$[v_6+\gamma u_6:u_6(w_6+1): u_6(w_6+1)^2]$, which is not defined as a point in $\mbP^2$ for $u_6=v_6=0$.

\subsection{Indices of polynomials with respect to blow-ups and degree computation}
\label{sect Example 2 degree computation}

Now the scene is set for the computation of the factorization of polynomials under the pull-back under $f$, and, as a consequence, for the computation of degrees of iterates of $f$. Since we only have one degree lowering component $\{y+z-w=0\}$, we introduce the following notion.

\begin{de}
Let $P(x,y,z,w)$ be a homogeneous polynomial. The index $\nu_2(P)$ of $P$ with respect to the blow-up map $\phi_2$ is defined by the relation
\begin{equation}\label{dP1 nu 2}
P\circ \phi_2=P(u_2,1,u_2v_2,u_2^2w_2)=u_2^{\nu_2(P)}\tilde{P}(u_2,v_2,w_2), \quad \tilde{P}(0,v_2,w_2)\not\equiv 0.
\end{equation}
\end{de}
The following result is a specialization of Theorem 4.5 from \cite{Alonso_Suris_Wei_2D}.
\begin{pro}\label{ex2 lemma1}
For a homogeneous polynomial $P$, its pull back $f^*P=P\circ f$ is divisible by $(y+z-w)^{\nu_2(P)}$ and by no higher degree of $y+z-w$.
\end{pro}
\begin{proof}
\begin{eqnarray*}
(f^*P)(x,y,z,w) & = & P\big(y(y+z-w),xw+\gamma z^2,z(y+z-w),(y+z-w)^2\big)\\
 & = & (xw+\gamma z^2)^d P(u_2,1,u_2v_2,u_2^2w_2),
\end{eqnarray*}
where
$$
u_2=\frac{y(y+z-w)}{xw+\gamma z^2}, \quad u_2v_2=\frac{z(y+z-w)}{xw+\gamma z^2}, \quad u_2^2w_2=\frac{(y+z-w)^2}{xw+\gamma z^2},
$$
so that
$$
v_2=\frac{z}{y}, \quad w_2=\frac{xw+\gamma z^2}{y^2}.
$$
Thus,
$$
(f^*P)(x,y,z,w)=(xw+\gamma z^2)^{d-\nu_2}\big(y(y+z-w)\big)^{\nu_2} \widetilde P\left(\frac{y(y+z-w)}{xw+\gamma z^2},\frac{z}{y}, \frac{xw+\gamma z^2}{y^2}\right).
$$
The polynomial on the right-hand side is manifestly divisible by $(y+z-w)^{\nu_2}$ and by no higher degree of $y+z-w$.
\end{proof}

Next, we study the behavior of a given polynomial $P=P_0$ under iterates of the pull-back $f^*$. We define recursively
\begin{equation}\label{ex2 P n+1}
P_{n+1}=(y+z-w)^{-\nu_2(P_n)}f^*P_n,
\end{equation} 
so that $P_{n+1}$ is the proper pull-back of $P_n$. We obtain recurrent relations for $\deg(n)$ and indices of $P_n$ with respect to the blow-up map $\phi_2$, as well as with respect to further blow-up maps $\phi_i$, $i=5,8,10$. The latter indices are defined in the standard way, by
$$
P\circ \phi_i=P(\phi_i(u_i,v_i,w_i))=u_i^{\nu_i(P)}\tilde{P}(u_i,v_i,w_i), \quad \tilde{P}(0,v_i,w_i)\not\equiv 0.
$$

\begin{theo}\label{ex2 Th degrees}
The degrees $d(n)=\deg(P_n)$ and indices $\nu_i(n)=\nu_i(P_n)$, $i=2,5,8,10$, satisfy the following system of recurrent relations:
\begin{align}
     d(n+1)&=2d(n)-\nu_2(n), \label{ex2 d n+1}\\
     \nu_2(n+1)&=\nu_5(n), \label{ex2 mu1 2}\\
     \nu_5(n+1)&=\nu_8(n),  \label{ex2 mu2 2}\\
     \nu_8(n+1)&=2 d(n)-2\nu_2(n)+\nu_{10}(n) , \label{ex2 mu3 2}\\
     \nu_{10}(n+1)&=2 d(n)-\nu_2(n). \label{ex2 mu4 2}
\end{align}
\end{theo}
\begin{proof} \mbox{}
\begin{itemize}
	\item Eq.\ \eqref{ex2 d n+1} follows directly from \eqref{ex2 P n+1}.
	\item Eqs.\ \eqref{ex2 mu1 2}, \eqref{ex2 mu2 2} hold true, since $f_X$ acts birationally in the neighborhood of the exceptional divisors $E_2$ and $E_5$ over the points $p_1=[0:1:0:0]$ and $p_2=[1:0:0:1]$, and the factor $y+z-w$ does not vanish at these points.
	\item For eq.\ \eqref{ex2 mu3 2}, we first prove that 
\begin{equation}\label{ex2 mu3 1}
\nu_8(f^*P_n)=2 d(n)+\nu_{10}(n).
\end{equation}
We compute:
\begin{eqnarray*}
&&(f^*P_n)(-\gamma u_8^2+v_8u_8^3,1-u_8+w_8u_8^2,u_8,1)  \\
&&\quad=P_n\big(w_8u_8^2(1-u_8+w_8u_8^2), v_8u_8^3, w_8u_8^3, w_8^2u_8^4\big)\\
 &&\quad=(u_8^2w_8(1-u_8+w_8u_8^2))^{d(P)}P_n(1,v_{10}u_{10}, u_{10}, w_{10}u_{10}^2)\\
 &&\quad\sim u_8^{2d(n)}P_n(1,v_{10}u_{10}, u_{10}, w_{10}u_{10}^2),
\end{eqnarray*}
where
$$
u_{10}=\frac{u_8}{1-u_8+w_8u_8^2}, \quad v_{10}u_{10}=\frac{(v_8/w_8)u_8}{1-u_8+w_8u_8^2}, \quad 
w_{10}u_{10}^2=\frac{w_8u_8^2}{1-u_8+w_8u_8^2},
$$
so that
$$
  v_{10}=\frac{v_8}{w_8}, \quad w_{10}=w_8(1-u_8+w_8u_8^2).
$$
Up to factor which does not vanish identically at $u_8=0$, the right-hand side behaves like $u_8^{2d(n)}u_{10}^{\nu_{10}(n)}\sim u_8^{2d(n)+\nu_{10}(n)}$. This proves \eqref{ex2 mu3 1}. Now \eqref{ex2 mu3 2} follows from \eqref{ex2 P n+1} by observing that $\nu_8(y+z-w)=2$.

\item Finally, for eq.\ \eqref{ex2 mu4 2}, we start with
\begin{equation}\label{ex2 mu4 1}
\nu_{10}(f^*P_n)=2 d(n). 
\end{equation}
For this, we compute:
\begin{small}
\begin{eqnarray*}
&&(f^*P_n)(1,v_{10}u_{10}, u_{10}, w_{10}u_{10}^2)  \\
 &&\quad\sim u_{10}^{2d(n)}P_n\big(v_{10} (1+v_{10}-u_{10}w_{10}),w_{10}+\gamma, 1+v_{10}-u_{10}w_{10},(1+v_{10}-u_{10}w_{10})^2\big).
\end{eqnarray*}
\end{small}Up to factor which does not vanish identically at $u_{10}=0$ (and represents the values of $P_n$ near the plane $\{x+z-w=0\}$), the right-hand side behaves like $u_{10}^{2d(n)}$. This proves \eqref{ex2 mu4 1}. Eq. \eqref{ex2 mu4 2} follows from \eqref{ex2 P n+1} by observing that $\nu_{10}(y+z-w)=1$.
\end{itemize}
\end{proof}

\begin{co}
For any polynomial P, the degrees of the polynomials $P_n$ satisfy the following recurrence relation:
\begin{equation}\label{ex2 d recur}
    d(n+1)=2d(n)-2d(n-2)+d(n-3).
\end{equation}
\end{co}
\begin{proof}
Comparing \eqref{ex2 d n+1} with \eqref{ex2 mu4 2}, we see that
\begin{equation}\label{ex2 mu4 = d}
\nu_{10}(n)=d(n), \quad n\ge 1.
\end{equation}
This allows us to re-write \eqref{ex2 mu3 2} as 
\begin{equation}\label{ex2 mu3 3}
\nu_8(n+1)=3 d(n)-2\nu_2(n), \quad n\ge 1.
\end{equation}
From \eqref{ex2 mu1 2}--\eqref{ex2 mu4 2}, \eqref{ex2 mu3 3} and \eqref{ex2 d n+1} we derive:
\begin{eqnarray*}
\nu_2(n)& = &  \nu_5(n-1)= \nu_8(n-2)\\
 & = & 3d(n-3)-2\nu_2(n-3)\\
 & = & 2\big(2d(n-3)-\nu_2(n-3)\big)-d(n-3)\\
 & = & 2d(n-2)-d(n-3).
\end{eqnarray*}
Plugging this into \eqref{ex2 d n+1} finally gives \eqref{ex2 d recur}.
\end{proof}

For a generic polynomial $P_0$ of degree 1 (i.e., a polynomial not passing through any of the points $p_1,\ldots,p_4$), we have:
$$
d(n)=n^2-n+2.
$$
The table of $d(n)$ and $\nu_i(n)$ in this case can be easily computed:
\smallskip

\begin{center}
\begin{tabular}[h]{c|c|cccc}
\hline
$n$ & $d(n)$ & $\nu_2(n)$ & $\nu_5(n)$ & $\nu_8(n)$ & $\nu_{10}(n)$ \\
\hline
0 & 1 & 0 & 0 & 0 & 0\\
1 & 2 & 0 & 0 & 2 & 2 \\
2& 4 & 0 & 2 & 6 & 4 \\
3 & 8 & 2 & 6 & 12 & 8 \\
4 & 14 & 6 & 12 & 20 & 14\\
5 & 22 & 12 & 20 & 30 & 22 \\
6 & 32 & 20 & 30 & 42 & 32 \\
7 & 44 & 30 & 42 & 56 & 44\\
8 & 58 & 42 & 56 & 72 & 58\\
9 & 74 & 56 & 72 & 90  & 74 \\
10 & 92 & 72 & 90 & 110 & 92
\end{tabular}
\end{center}

\subsection{Invariant foliation and integrals of motion}
\label{sect Example 2 inv}

To find invariant foliations of $f$, we propose a method based on indices of polynomials at blow-ups. We formulate the following definition for our current example, but it is of a general character, provided one has a lift $f_X$ resolving the contraction of all critical components of $f$, and the indices at corresponding blow-ups. Recall that a polynomial $P$ is called a Darboux polynomial \cite{Celledoni4}, \cite{Gasull2002}, \cite{Goriely2001}, if $f^*P$ is divisible by $P$, which in our case means simply $f^*P=(y+z-w)^{\nu_2(P)}P$, or, in other words, if $\widehat P=P$, where $\widehat{P}:=(y+z-w)^{-\nu_2(P)}f^*P$ is the proper pull-back of $P$. We introduce a somewhat weaker notion. If $f_X$ is algebraically stable, we define a \emph{generalized Darboux polynomial} by the property that the divisor class of $\{P=0\}$ in ${\rm Pic}(X)$ is fixed under the action of $f^*_X$. If $f_X$ is not algebraically stable, like in the present case, the following definition will serve as a replacement.

\begin{de}
We say that a polynomial $P$ is a \emph{generalized Darboux polynomial} for the map $f$ if its proper pull-back $\widehat{P}:=(y+z-w)^{-\nu_2(P)}f^*P$ satisfies 
\begin{equation}\label{ex2 inv foliation}
d(\widehat{P})=d(P) \quad \text{and}\quad \nu_i(\widehat{P})=\nu_i(P), \quad i=2,5,8,10. 
\end{equation}
\end{de}
Obviously, generalized Darboux polynomials (of a given degree) form a vector space. The computation of generalized Darboux polynomials in the present case is facilitated by observing that, by Theorem \ref{ex2 Th degrees}, condition \eqref{ex2 inv foliation} implies $\nu_i(P)=d(P)$ for $i=2,5,8,10$.

For $d(P)=1$, there is only one such polynomial, $P=z$. For $d(P)=2$, one finds two linearly independent polynomials,
$Q_0=w(x+y+z-w)$ and $Q_\infty=z^2$, with $f^*Q_0=(y+z-w)^2(Q_0+\gamma Q_\infty)$ and $f^*Q_\infty=(y+z-w)^2Q_\infty$. In other words, we have a
pencil of quadrics which are generalized Darboux polynomials, $Q_\lambda=w(x+y+z-w)+\lambda z^2$, $\lambda\in\mbC$, with 
$$
f^*Q_\lambda=(y+z-w)^2 Q_{\lambda+\gamma}.
$$ 
The foliation of $\mbP^3$ by the quadrics $\{Q_\lambda=0\}$ is preserved under $f$, but no quadric except $Q_\infty$ is preserved individually, unless $\gamma=0$. 

For $\gamma=0$ (an autonomous version of dPI), however, not only are the quadrics $\{Q_\lambda=0\}$ preserved individually, but the dynamics changes drastically also in other respects. Both indeterminacy sets $\mathcal I(f)$ and $\mathcal I(f^{-1})$ degenerate into pairs of coplanar lines:
$$
\mathcal I(f)=\ell_1\cup \ell_2,\;\;\text{where} \;\; \ell_1=\{y+z-w=0\}\cap \{x=0\}, \;\; \ell_2=\{y+z-w=0\}\cap \{w=0\},
$$ 
$$
\mathcal I(f^{-1})=\ell_3\cup \ell_4,\;\; \text{where}\;\; \ell_3=\{x+z-w=0\}\cap \{y=0\}, \;\; \ell_4=\{x+z-w=0\}\cap \{w=0\}.
$$

In this case, the plane $\{w=0\}$ is also degree lowering. We have the following degree lowering sequence:
$$
\{w=0\}\;\;\to \;\;\ell_3 \;\; \to \;\; \ell_1 \;\;\to\;\; \{w=0\}.
$$
So, this case does not fit our standing assumption that all degree lowering surfaces are contracted to points. The whole application of our algorithm has to be reworked. On this way, we find a second pencil of quadratic polynomials such that the degrees and the relevant indices for the proper transforms $\widehat P$ coincide with those for $P$, namely 
$\{P_\mu:=z(x+y+z-w)+\mu xy\}$, for which 
$$
f^*P_\mu=(y+z-w)wP_\mu.
$$
Thus, for $\gamma=0$, the map $f$ has two integrals of motion, $\lambda=-w(x+y+z-w)/z^2$ and $\mu=-z(x+y+z-w)/(xy)$. Note that, according to \eqref{dP1 change inv}, in the original variables $(u,v,c)$ of the dPI equation, the integral $\lambda$ corresponds to the variable $c$ (constant $c_0$ in the case $\gamma=0$), while the integral $\mu$ corresponds to the integral $(u+v-1)(c_0-uv)$ of the map, that is, to the integral $(u_n+u_{n-1}-1)(c_0-u_nu_{n-1})$ of the difference dPI equation \eqref{dP1}.

\subsection{An algebraically stable lift}

 We remark here that there exists an algebraically stable lift of $f$, which is obtained by a different sequence of blowing ups on $\mbP^3$. 

\begin{enumerate}
    \item    Blow-up at the point $p_1=[0:1:0:0]$ in $\mbP^3$,
    $$
   \pi_1:\; \left\{\begin{aligned} 
     x&=u_1,\\y&=1,\\z&=u_1v_1,\\w&=u_1w_1.\end{aligned}\right.
   $$
   The exceptional divisor of this blow-up is given by $E_1=\{u_1=0\}$.
   \item Blow-up at the curve $\{w_1=0\}$ in $E_1$,
         $$
\pi_2:\,\left\{\begin{aligned} u_1&=u_2, \\ v_1&=v_2, \\w_1&=w_2u_2,  \end{aligned}\right. \quad \Rightarrow\quad 
\phi_2=\pi_1\circ\pi_2:\,\left\{\begin{aligned} x&=u_2, \\ y&=1,\\ z&=u_2v_2, \\ w&=u_2^2w_2. \end{aligned}\right. 
    $$
   The exceptional divisor of this blow-up is given by $E_2=\{u_2=0\}$.
   \item Blow-up at the point $p_2=[1:0:0:1]$ in $\mbP^3$,
     $$
  \pi_3:\,  \left\{\begin{aligned} x&=1,\\ y&=u_3v_3, \\ z&=u_3, \\ w&=1+u_3w_3.\end{aligned}\right.
    $$
     The exceptional divisor of this blow-up is given by $E_3=\{u_3=0\}$.
  \item  Blow-up at the curve $\{v_3=0\}$ in $E_3$,
           $$
\rho_4:\;\left\{\begin{aligned} u_3&=u_4, \\ v_3&=u_4v_4, \\w_3&=w_4,  \end{aligned}\right. \quad \Rightarrow\quad 
\pi_3\circ\rho_4:\,\left\{\begin{aligned} x&=1, \\ y&=u_4^2v_4,\\ z&=u_4, \\ w&=1+u_4w_4. \end{aligned}\right. 
    $$
   The exceptional divisor of this blow-up is given by $F_4=\{u_4=0\}$.
  \item Blow-up at the point $(u_4,v_4,w_4)=(0,\gamma,1)$ in $F_4$,
             $$
\rho_5:\, \left\{\begin{aligned} u_4&=u_5, \\ v_4&=\gamma+u_5v_5, \\w_4&=1+u_5w_5,  \end{aligned}\right. \quad \Rightarrow\quad 
\psi_5=\pi_3\circ\rho_4\circ\rho_5:\,\left\{\begin{aligned} x&=1, \\ y&=\gamma u_5^2+u_5^3v_5,\\ z&=u_5, \\ w&=1+u_5+u_5^2w_5. \end{aligned}\right. 
    $$
   The exceptional divisor of this blow-up is given by $F_5=\{u_5=0\}$.
   \item Blowing up at the point $p_3=[0:1:0:1]$ in $\mbP^3$,
      $$
    \pi_6:\, \left\{\begin{aligned} x&=u_6v_6,\\ y&=1+u_6w_6, \\z&=u_6, \\ w&=1.\end{aligned}\right.
     $$
     The exceptional divisor of this blow-up is given by $E_6=\{u_6=0\}$.
    \item Blow-up at the curve $\{v_6=0\}$ in $E_6$,
               $$
\rho_7:\,\left\{\begin{aligned} u_6&=u_7, \\ v_6&=u_7v_7, \\w_6&=w_7,  \end{aligned}\right. \quad \Rightarrow\quad 
\pi_6\circ\rho_7:\,\left\{\begin{aligned} x&=u_7^2v_7, \\ y&=1+u_7w_7,\\ z&=u_7, \\ w&=1. \end{aligned}\right. 
    $$
   The exceptional divisor of this blow-up is given by $F_7=\{u_7=0\}$.
   \item Blow-up at the point $(u_7,v_7,w_7)=(0,-\gamma,-1)$ in $F_7$,
                $$
\rho_8:\,\left\{\begin{aligned} u_7&=u_8, \\ v_7&=-\gamma+u_8v_8, \\w_7&=-1+u_8w_8,  \end{aligned}\right. \quad \Rightarrow\quad 
\psi_8=\pi_6\circ\rho_7\circ\rho_8:\,\left\{\begin{aligned} x&=-\gamma u_8^2+u_8^3v_8, \\ y&=1-u_8+u_8^2w_8,\\ z&=u_8, \\ w&=1. \end{aligned}\right. 
    $$
   The exceptional divisor of this blow-up is given by $F_8=\{u_8=0\}$.
   \item Blow-up at the point $p_4=[1:0:0:0]$ in $\mbP^3$,
     $$
    \pi_9:\,\left\{\begin{aligned} x&=1,\\ y&=u_9v_9,\\ z&=u_9,\\ w&=u_9w_9.\end{aligned}\right.
    $$
      The exceptional divisor of this blow-up is given by $E_9=\{u_9=0\}$.
    \item  Blow-up at the curve $\{w_9=0\}$ in $E_9$,
    $$
   \pi_{10}:\,  \left\{\begin{aligned} u_9&=u_{10}, \\ v_9&=v_{10}, \\ w_9&=w_{10}u_{10}, \end{aligned}\right. \quad \Rightarrow \quad 
     \phi_{10}=\pi_9\circ\pi_{10}:\,\left\{\begin{aligned} x&=1,\\ y&=u_{10}v_{10}, \\ z&=u_{10}, \\ w&=u_{10}^2w_{10}. \end{aligned}\right.    
    $$
    The exceptional divisor of this blow-up is given by $E_{10}=\{u_{10}=0\}$.
    \item Blow-up at the line $\{z=w=0\}$ in $\mbP^3$,
         $$
  \psi_{11}:\;  \left\{\begin{aligned} x&=1,\\ y&=v_{11},\\ z&=u_{11},\\ w&=u_{11}w_{11}.\end{aligned}\right.
    $$
      The exceptional divisor of this blow-up is given by $F_{11}=\{u_{11}=0\}$.
\end{enumerate}
We note that the blow-ups $\psi_5$ and $\psi_8$ are different from the blow-ups $\phi_5$ and $\phi_8$ of Section \ref{sect Example2 blow-up 1}, even if they are expressed by identical formulas in the corresponding local affine coordinate patches. The easiest way to see this is to observe that the exceptional divisors $E_5$ and $E_8$ are isomorphic to $\mbP^1\times\mbP^1$ (the top elementary blow-ups are centered at curves), while the exceptional divisors $F_5$ and $F_8$ are isomorphic to $\mbP^2$ (the top elementary blow-ups are centered at points). 

We denote the blowing up variety by $Y$ and the corresponding lifted map by $f_Y$. The action of $f_Y$ on divisors can be summarized as follows, see Figure \ref{fig:2}:
\begin{eqnarray}
    & \overline{\{y+z-w=0\}}\;\dasharrow\; E_2\; \dasharrow \; F_5 \; \dasharrow \; F_8 \; \dasharrow \;  E_{10}\; \dasharrow  \; \overline{\{x+z-w=0\}}, & \\
     & E_1 \; \dasharrow \;  F_4 \; \dasharrow \;  F_7\; \dasharrow \;  E_9 \; \dasharrow \;  E_1, & \qquad \qquad\\
    & E_3\; \dasharrow \;  E_6 \; \dasharrow \;  F_{11} \; \dasharrow \;  E_3. &
\end{eqnarray}

\begin{figure}[htb]
        \centering
        \input{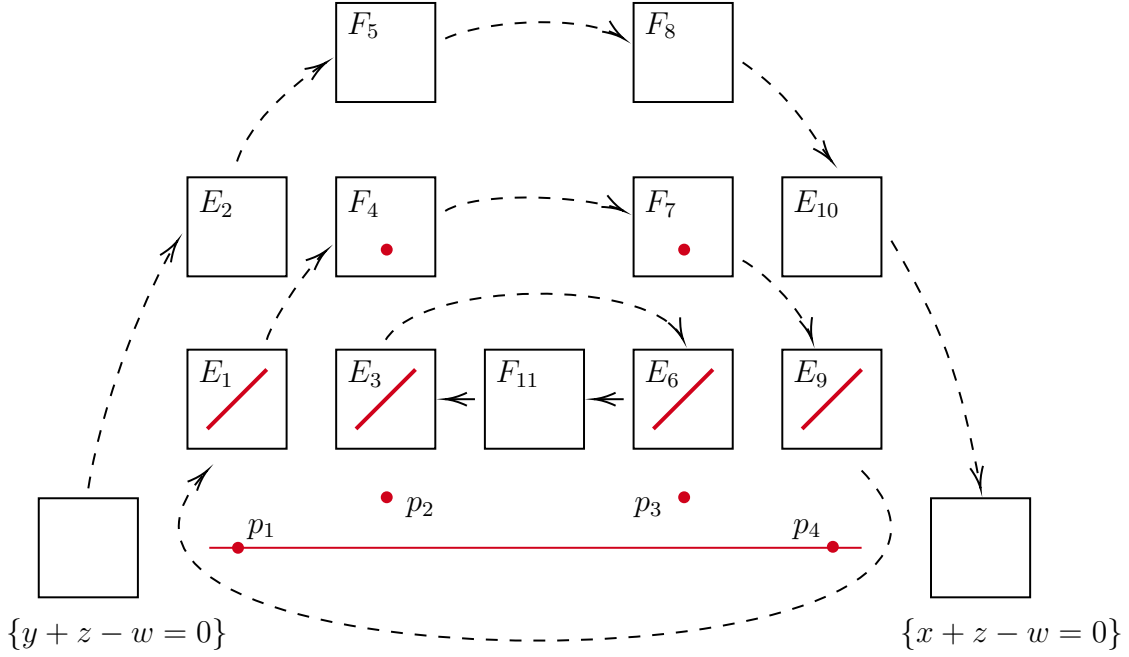}
        \caption{An algebraically stable lift. The same conventions as on Figure \ref{fig:1} are used.}
        \label{fig:2}
\end{figure}

One can easily see that $f_Y$ is algebraically stable, since every exceptional divisor in the diagram has both a birational image and a birational preimage. 

In order to compute the action of $(f_Y)_*$ on the Picard group ${\rm Pic}(Y)$, it is more convenient to use total transforms $\cF_i$ of the blow-up centers rather than proper transforms $E_i$, $F_i$. The relation between the former and the latter is the following:
     $$
    \left\{\begin{aligned} E_1&=\cF_1-\cF_2,\\ E_2&=\cF_2,\\ E_3&=\cF_3-\cF_4,\\ F_4&=\cF_4-\cF_5,
    \\ F_5&=\cF_5, \\ E_6&=\cF_6-\cF_7,\\ F_7&=\cF_7-\cF_8, \\
     F_8&=\cF_8\\ E_9&=\cF_9-\cF_{10},\\ E_{10}&=\cF_{10},\\ F_{11}&=\cF_{11}\end{aligned}\right.
    $$
We determine as follows the divisor classes of the proper transforms  $\overline{\{x+z-w=0\}}$ and  $\overline{\{y+z-w=0\}}$. The plane $\{x+z-w=0\}$ passes through the points $p_1$ and $p_2$. The intersection of $\overline{\{x+z-w=0\}}$ and the exceptional divisor $\cF_1$ is given by $\{1+v_1-w_1=0\}$, different from the curve $\{w_1=0\}$ which is the center of the blow-up $\pi_2$. The intersection of $\overline{\{x+z-w=0\}}$ and the exceptional divisor $\cF_3$ is given by $\{1-w_3=0\}$, different from the curve $\{v_3=0\}$ which is the center of the blow-up $\rho_4$. The intersection of $\overline{\{x+z-w=0\}}$ and the exceptional divisor $\cF_4$ is given by $\{1-w_4=0\}$, which contains the point $(u_4,v_4,w_4)=(0,\gamma,1)$, the center of the blow-up $\rho_5$.  Therefore, divisor class of the proper transform $\overline{\{x+z-w=0\}}$ is $H-\cF_1 -\cF_3 -\cF_5$. 

Similarly, the divisor class of the proper transform $\overline{\{y+z-w=0\}}$ is $H-\cF_6 -\cF_8 -\cF_9$.

With similar arguments one computes the action of $(f_Y)_*$ on ${\rm Pic}(Y)$
with respect to the basis $\{H,\cF_1,\ldots,\cF_{11}\}$:
\[ 
\left( \begin{array}{cccccccccccc}
H  \\
\cF_1  \\
\cF_2 \\
\cF_3\\
\cF_4 \\
\cF_5\\
\cF_6 \\
\cF_7 \\
\cF_8\\
\cF_9\\
\cF_{10}\\
\cF_{11}
\end{array} \right)
\rightarrow
\left( \begin{array}{cccccccccccc}
2H-\cF_1 -\cF_2 -\cF_3 -\cF_5 \\
\cF_4\\
\cF_5\\
\cF_6\\
\cF_7\\
\cF_8\\
\cF_9+\cF_{11} \\
\cF_9\\
\cF_{10}\\
H-\cF_2 -\cF_3 -\cF_5  \\
H-\cF_1 -\cF_3 -\cF_5  \\
\cF_3 -\cF_4 
\end{array} \right).
\]
This allows us for an alternative computation of the degrees $d(n)=\deg f^n$ as the $(H,H)$ entry of the matrix $((f_Y)_*)^n$, resulting in the same formula $d(n)=n^2-n+2$ as in Section \ref{sect Example 2 degree computation}.

Another application of this algebraically stable lift and its Picard group is a computation of the generalized Darboux polynomials, alternative to the one from Section \ref{sect Example 2 inv}. This time, we can define generalized Darboux polynomials $P$ by the invariance of $\overline{\{P=0\}}$ under $f_Y^*$. The eigenspace of $(f_Y)_*$ for the eigenvalue 1 contains an effective divisor class $H-\cF_1-\cF_3-\cF_6-\cF_9-\cF_{11}$ with a one-dimensional linear system corresponding to the linear polynomial $P=z$, as well as an effective divisor class  
$$
2H-\cF_1 -\cF_2 -\cF_3 -\cF_5-\cF_6 -\cF_8 -\cF_9 -\cF_{10}- \cF_{11}
$$
with a two-dimensional linear system corresponding to the pencil of quadrics 
$$
Q_\lambda=\{w(x+y+z-w)+\lambda z^2=0\},
$$ 
discussed in Section \ref{sect Example 2 inv}.


\section{Example 3: a 3D map obtained by inflation of a QRT map}
\label{sect Example 3}

\subsection{Introduction}
This example is from \cite{Viallet_2019}, which provides a birational map $f$ on $\mbP^3$ with cubic degree growth. The map is obtained from a second order difference equation
\begin{equation}\label{eq QRT}
\xi_{n+1}+\xi_{n-1}=-\frac{d\xi_n^2+b\xi_n+a}{d\xi_n+c}
\end{equation}
via a non-invertible ``inflation transformation'' $\xi_n=\eta_n+\eta_{n-1}$. The resulting third order difference equation gives in homogeneous coordinates the following birational map of $\mbP^3$:
\begin{equation}\label{ex3 f}
 f: \left[\begin{array}{c} x\\y\\z\\w\end{array}\right]\mapsto\left[\begin{array}{c} -d(x+y)(2x+2y+z)-bw(x+y)-cw(x+y+z)-aw^2 \\ x\big(d(x+y)+cw\big) \\ y\big(d(x+y)+cw\big)\\w\big(d(x+y)+cw\big)\end{array}\right].
\end{equation}
The inverse map is given by
\begin{equation}\label{ex3 f inv}
 f^{-1}: \left[\begin{array}{c} x\\y\\z\\w\end{array}\right]\mapsto\left[\begin{array}{c} y\big(d(y+z)+cw\big) \\ z\big(d(y+z)+cw\big) \\ -d(y+z)(x+2y+2z)-bw(y+z)-cw(x+y+z)-aw^2 \\ w\big(d(y+z)+cw\big)\end{array}\right].
\end{equation}

Under the condition $ad+c^2-bc\neq 0$, which we always assume, the indeterminacy set of $f$ is the line
\begin{equation}\label{ex3 I(f)}
{\mathcal I}(f)=\{x+y=0, \; w=0\}.
\end{equation}
Likewise, the indeterminacy set of $f^{-1}$ is the line
\begin{equation}\label{ex3 I(f inv)}
{\mathcal I}(f^{-1})=\{y+z=0,\; w=0\}.
\end{equation}
The critical sets of the maps $f$ and $f^{-1}$ are the planes
\begin{equation}\label{ex3 C(f)}
{\mathcal C}(f)=\{d(x+y)+cw=0\},
\end{equation}
resp. 
\begin{equation}\label{ex3 C(f inv)}
{\mathcal C}(f^{-1})=\{d(y+z)+cw=0\}.
\end{equation}

The plane $\{d(x+y)+cw=0\}$ is blown down to a point whose further iterates land at ${\mathcal I}(f)$:
\begin{equation}\label{eq Example3 confined}
f:\; {\mathcal C}(f)= \{d(x+y)+cw=0\} \; \to \; p_1 \; \to \; p_2\; \to \; p_3\in {\mathcal I}(f),
\end{equation}
where
$$
p_1=[1:0:0:0], \;\; p_2=[2:-1:0:0],\;\; p_3=[2:-2:1:0].
$$

We will regularize this map following the prescription from \cite{Alonso_Suris_Wei_2D}, via blowing up the orbit of $p_1$, and then we will compute the degrees of $f^n$ with the help of local indices of polynomials at the relevant blow-ups.  
We will see that, via a regularization of the map $f$, the confinement pattern \eqref{eq Example3 confined} can be, in a certain sense, extended beyond $p_3$:
\begin{equation}\label{eq Example3 long confined}
f:\; {\mathcal C}(f)= \{d(x+y)+cw=0\}  \; \to \; p_1 \; \to \; p_2\; \to \; p_3\; \to \; p_4\;\;  \toitself
\end{equation}
where
$$
p_4=[-1:1:-1:0]\in{\mathcal I}(f).
$$
The precise meaning of $p_4$ being a fixed singular point will be clarified in the following discussion.

\subsection{Regularization of the map $f$ by blow-ups}
\label{sect Example3 reg}

\paragraph{Action of $f$ on the critical plane $\{d(x+y)+cw=0\}$.} The critical plane ${\mathcal C}(f)=\{d(x+y)+cw=0\}$ is contracted by $f$ to the point $p_1$. We resolve this with a sequence of three blow-ups at $p_1$.
\begin{itemize}
    \item At the first step, we blow up the point $p_1=[1:0:0:0]\in \mbP^3$: 
\begin{equation}
\pi_1:\, \left\{\begin{aligned} 
x&=1,\\y&=u_1v_1,\\ z&=u_1w_1,\\w&=u_1.\end{aligned}\right.
\end{equation}
The exceptional divisor of $\pi_1$ is given by $E_1=\{u_1=0\}$. The image of the (proper transform of the)  plane $\{d(x+y)+cw=0\}$ in the exceptional divisor is the curve  $\{d(v_1+w_1)+c=0\}\subset E_1$.
    \item At the second step, we blow up this curve in $E_1$:
$$
\pi_2:\; \left\{\begin{aligned} u_1&=u_2, \\ v_1&=v_2, \\w_1&=-v_2-\dfrac{c}{d}+w_2u_2,  \end{aligned}\right. \quad \Rightarrow\quad 
\pi_1\circ \pi_2:\, \left\{\begin{aligned} x&=1, \\ y&=u_2v_2,\\ z&=-u_2\Big(v_2+\dfrac{c}{d}\Big)+ u_2^2w_2, \\ w&=u_2. \end{aligned}\right. 
$$
The exceptional divisor of this blow-up is given by $E_2=\{u_2=0\}$. The image of the (proper transform of the)  plane $\{d(x+y)+cw=0\}$ in $E_2$ is the curve $\{w_2=-(ad+c^2-bc)/d^2\}\subset E_2$.
    \item At the third step, we blow up the latter curve:
$$
\pi_3:\,\left\{\begin{aligned} u_2&=u_3, \\ v_2&=v_3, \\ w_2&=-\dfrac{ad+c^2-bc}{d^2}+w_3u_3,  \end{aligned}\right. 
$$
\begin{equation}\label{eq E3}
\qquad \Rightarrow\quad 
\phi_3=\pi_1\circ\pi_2\circ\pi_3:\,\left\{\begin{aligned} x&=1, \\ y&=u_3v_3,\\ z&=-\Big(v_3+\dfrac{c}{d}\Big)u_3-\dfrac{ad+c^2-bc}{d^2}u_3^2+w_3u_3^3, \\ w&=u_3. \end{aligned}\right. 
\end{equation}
The exceptional divisor of this blow-up is given by $E_3=\{u_3=0\}$. The (proper transform of the)  plane $\{d(x+y)+cw=0\}$ is mapped by the lift of $f$ birationally to $E_3$.
\end{itemize}

\paragraph{Action of $f$ on $E_3$.}
Next, we consider the fate of the exceptional divisor $E_3$ under $f$. The lift of $f$ under $\phi_3$ contracts $E_3$ to the point $p_2$. We resolve this by a sequence of three blow-ups at $p_2$. 
\begin{itemize}
    \item At the first step, we blow up the point $p_2=[-2:1:0:0]\in \mbP^3$: 
    $$
    \pi_4:\,\left\{\begin{aligned} x&=-2+u_4w_4,\\ y&=1, \\ z&=u_4v_4, \\ w&=u_4.\end{aligned}\right.
    $$
The exceptional divisor of $\pi_4$ is given by $E_4=\{u_4=0\}$. The image of $E_3$ under the lift of $f$ is the curve $\Big\{v_4+w_4=-\dfrac{b-2c}{d}\Big\}\subset E_4$.
    \item At the second step, we blow up this curve:
    $$
     \pi_5:\,\left\{\begin{aligned} u_4&=u_5,\\ v_4&=v_5, \\ w_4&=-v_5-\dfrac{b-2c}{d}+u_5w_5,  \end{aligned}\right. \!\! \Rightarrow \;
     \pi_4\circ\pi_5:\,\left\{\begin{aligned} x&=-2-u_5\Big(v_5+\dfrac{b-2c}{d}\Big)+u_5^2w_5,\\ y&=1, \\ z&=u_5v_5, \\ w&=u_5.\end{aligned}\right.
     $$
The exceptional divisor of $\pi_5$ is given by  $E_5=\{u_5=0\}$. The image of $E_3$ under the lift of $f$ is the curve $\{w_5=0\}$ in $E_5$.
     \item At the third step, we blow up the latter curve:
$$
  \pi_6:\,\left\{\begin{aligned} u_5&=u_6,\\ v_5&=v_6, \\ w_5&=u_6w_6,  \end{aligned}\right. \hspace{3.5cm}
$$
\begin{equation}\label{eq E6}
  \quad \Rightarrow \quad 
   \phi_6=\pi_4\circ\pi_5\circ\pi_6:\,\left\{\begin{aligned} x&=-2-u_6\Big(v_6+\dfrac{b-2c}{d}\Big)+u_6^3w_6,\\ y&=1, \\ z&=u_6v_6, \\ w&=u_6.
  \end{aligned}\right.
\end{equation}
The exceptional divisor of $\pi_6$ is given by  $E_6=\{u_6=0\}$, and the lift of $f$ by $\phi_6$ maps $E_3$ to $E_6$ birationally.
\end{itemize}

\paragraph{Action of $f$ on $E_6$.} The lift of $f$ under $\phi_6$ contracts $E_6$ to the point $p_3$. To resolve this, we perform a sequence of three blow-ups at $p_3$. 
\begin{itemize}
    \item We first blow up the point $p_3=[2:-2:1:0]\in\mbP^3$:
     $$
     \pi_7:\,\left\{\begin{aligned} x&=2+u_7v_7,\\ y&=-2+u_7w_7, \\ z&=1, \\ w&=u_7.\end{aligned}\right.
     $$ 
The exceptional divisor of this blow-up is  $E_7=\{u_7=0\}$. The image of $E_6$ under the lift of $f$ by $\pi_7$ is the curve $\Big\{v_7+w_7=-\dfrac{c}{d}\Big\}\subset E_7$.
     \item Then we blow up this curve:
     $$
   \pi_8:\, \left\{\begin{aligned} u_7&=u_8,\\ v_7&=v_8, \\ w_7&=-v_8-\dfrac{c}{d}+u_8w_8,  \end{aligned}\right. \quad \Rightarrow \quad 
   \pi_7\circ\pi_8:\,\left\{\begin{aligned} x&=2+u_8v_8,\\ y&=-2-u_8\Big(v_8+\dfrac{c}{d}\Big)+u_8^2w_8, \\ z&=1, \\ 
   w&=u_8.\end{aligned}\right.
    $$
The exceptional divisor of this blow-up is  $E_8=\{u_8=0\}$. The image of $E_6$ under the lift of $f$ by $\pi_7\circ\pi_8$ is the curve $\Big\{w_8=\dfrac{1}{d^2}(ad+c^2-bc)\Big\}\subset E_8$.
    \item Next, we blow up the latter curve:
$$
\pi_9:\,\left\{\begin{aligned} u_8&=u_9, \\ v_8&=v_9,\\ w_8&=\dfrac{1}{d^2}(ad+c^2-bc)+u_9w_9,  \end{aligned}\right.
$$
\begin{equation}\label{eq E9}
 \phi_9=\pi_7\circ\pi_8\circ\pi_9:\,  \quad \Rightarrow \quad
    \left\{\begin{aligned} x&=2+u_9v_9,\\ y&=-2-u_9\Big(v_9+\dfrac{c}{d}\Big)+\dfrac{1}{d^2}(ad+c^2-bc)u_9^2+w_9u_9^3, \\ 
  z&=1, \\ w&=u_9.\end{aligned}\right.
\end{equation}
The exceptional divisor of $\pi_9$ is given by  $E_9=\{u_9=0\}$, and the lift of $f$ by $\phi_9$ maps $E_6$ to $E_9$ birationally.
\end{itemize}

\paragraph{Action of $f$ on $E_9$.} The lift of $f$ by $\phi_9$ contracts $E_9$ to the point $p_4=[-1:1:-1:0]\in\mbP^3$. Indeed, a straightforward computation of the lifted action of $f$ on the coordinate patch \eqref{eq E9} leads to
\begin{equation}\label{eq B3}
   \tilde f\circ \phi_9=u_9^2\left(\begin{array}{c} 2d^2(ad+c^2-bc)+O(u_9) \\ -2d^2(ad+c^2-bc)+O(u_9) \\ 2d^2(ad+c^2-bc)+O(u_9) \\ -d^2u_9(ad+c^2-bc+d^2u_9w_9) \end{array}\right).
\end{equation}
To resolve this, we perform a blow-up at the point $p_4$:
\begin{equation}\label{eq E10}
    \phi_{10}:\,\left\{\begin{aligned} x&=-1+u_{10}v_{10},\\ y&=1,\\ z&=-1+u_{10}w_{10},\\ w&=u_{10}.\end{aligned}\right.
\end{equation}
The exceptional divisor of this blow-up is given by $E_{10}=\{u_{10}=0\}$. The image of $E_9$ under the lift of $f$ by $\phi_{10}$ is the curve
\begin{equation}\label{image in E10}
\Big\{w_{10}=-\dfrac{c}{d}\Big\} \subset E_{10}.
\end{equation} 

Compute the lifted action of $f$ on the coordinate patch \eqref{eq E10}:
\begin{eqnarray}\label{eq f phi10}
\tilde f\circ \phi_{10} & = & u_{10}\begin{pmatrix} (dv_{10}+c) -u_{10}(dv_{10}+c)(v_{10}+w_{10})-u_{10}(dv_{10}^2+bv_{10}+a)
 \\  (dv_{10}+c)(-1+u_{10}v_{10})  \\ dv_{10}+c \\ u_{10}(dv_{10}+c)\end{pmatrix} \qquad\nonumber\\
 & = &  
u_{10}\begin{pmatrix} -1+\widetilde{u}_{10}\widetilde{v}_{10} \\  1  \\ -1+\widetilde{u}_{10}\widetilde{w}_{10} \\ \widetilde{u}_{10}\end{pmatrix} = u_{10}\,\widetilde{\phi}_{10}(\widetilde{u}_{10},\widetilde{v}_{10},\widetilde{w}_{10}),
\end{eqnarray}
where
$$
\widetilde{u}_{10}=-\frac{u_{10}}{1-u_{10}v_{10}}\sim u_{10},
$$
and 
\begin{equation}\label{QRT in E10}
\widetilde{v}_{10}=-w_{10}-\frac{dv_{10}^2+bv_{10}+a}{dv_{10}+c}, \quad \widetilde{w}_{10}=v_{10}.
\end{equation}
We see that the curve 
\begin{equation}\label{singular curve in E10}
\left\{v_{10}=-\frac{c}{d}\right\}\subset E_{10}=\{u_{10}=0\}
\end{equation} 
belongs to the indeterminacy set of the lifted map $f$, but away from this curve $E_{10}$ is invariant, and the restriction of the lifted map $f$ to $E_{10}\setminus \{v_{10}=-c/d\}$ is given by \eqref{QRT in E10},  which is nothing but the original QRT map \eqref{eq QRT}. This observation was made in \cite{Viallet_2019}.

In order to completely regularize the map $f$, we would need to blow-up the curve  \eqref{image in E10} and all its further images in $E_{10}$. However, this turns out to be not necessary for the degree computations. Indeed, the successive images of the curve  \eqref{image in E10} are irreducible curves in $E_{10}$, not containing the singular curve \eqref{singular curve in E10}. Thus, the lift $f_X$ of $f$ to the variety $X$ obtained from $\mbP^3$ by the blow-ups $\phi_3$ at $p_1$, $\phi_6$ at $p_2$, $\phi_9$ at $p_3$ and $\phi_{10}$ at $p_4$ is algebraically stable, an it suffices to use the indices of polynomials with respect to these blow-ups. Compare Example 3 in \cite{Alonso_Suris_Wei_2D}. The action of this lift can be summarized as follows:
\begin{equation}
f_X:\;  \overline{\{d(x+y)+cw=0\}}  \; \dasharrow \; E_3 \; \dasharrow \; E_6\; \dasharrow \; E_9\; \dasharrow \; E_{10}\;\;  \toitself
\end{equation}
which is an explanation of equation \eqref{eq Example3 long confined}.

\subsection{Indices of polynomials with respect to blow-ups and degree computations}

We proceed with defining the indices of polynomials with respect to the above mentioned blow-ups.

\begin{de}
Let $P(x,y,z,w)$ be a homogeneous polynomial. The index $\nu_3(P)$ with respect to the blow-up map $\phi_3$ from \eqref{eq E3} is defined by the relation
$$
P\circ\phi_3=u_3^{\nu_3(P)}\widetilde{P}(u_3,v_3,w_3), \quad \widetilde{P}(0,v_3,w_3)\not\equiv 0.
$$
\end{de}
Similarly one defines the indices $\nu_6(P)$, $\nu_9(P)$, $\nu_{10}(P)$ with respect to the blow-ups $\phi_6$ from  \eqref{eq E6}, $\phi_9$ from \eqref{eq E9}, and $\phi_{10}$ from \eqref{eq E10}.

The following result is a specialization of Theorem 4.5 from \cite{Alonso_Suris_Wei_2D}.
\begin{pro}\label{ex3 lemma1}
For a homogeneous polynomial $P$, its pull back $f^*P=P\circ f$ is divisible by $(d(x+y)+cw)^{\nu_3(P)}$ and by no higher degree of $d(x+y)+cw$.
\end{pro}

\begin{proof}
The proof is essentially the same as for Proposition \ref{ex2 lemma1}.
\end{proof}

As usual, for any homogeneous polynomial $P=P_0$, we define iterated proper pull-backs $P_{n+1}$ by 
\begin{equation}\label{P n+1}
P_{n+1}=\big(d(x+y)+cw\big)^{-\nu_3(P_n)}(f^*P_n),
\end{equation} 
and set $d(n)=\deg(P_n)$ and $\nu_i(n)=\nu_i(P_n)$.

\begin{theo}
These quantities satisfy the following recurrent relations:
\begin{align}
    d(n+1)&=2d(n)-\nu_3(n), \label{d n+1}\\
     \nu_3(n+1)&=\nu_6(n), \label{mu1 2}\\
     \nu_6(n+1)&=\nu_9(n),  \label{mu2 2}\\
     \nu_9(n+1)&=2 d(n)-2\nu_3(n)+\nu_{10}(n) , \label{mu3 2}\\
     \nu_{10}(n+1)&= d(n)-\nu_3(n)+\nu_{10}(n). \label{mu4 2}
\end{align}
\end{theo}
\begin{proof}
Equation \eqref{d n+1} follows immediately from \eqref{P n+1}. 

Equations \eqref{mu1 2}, \eqref{mu2 2} follow from the fact that $f_X$ acts birationally in the neighborhoods of the exceptional divisors $E_3$ and $E_6$ of the blow-ups $\phi_3$ and $\phi_6$ over the points $p_1=[1:0:0:0]$ and $p_2=[2:-1:0:0]$, respectively, and the factor $d(x+y)+cw$ does not vanish at these points. 

To prove \eqref{mu3 2}, we derive from \eqref{eq B3}:
$$
\nu_9(f^*P_n)=2 d(n)+\nu_{10}(n),
$$
and then the result follows from \eqref{P n+1} by using $\nu_9\big(d(x+y)+cw\big)=2$.

Similarly, to prove \eqref{mu4 2}, we derive from \eqref{eq f phi10}:
$$
\nu_{10}(f^*P_n)=d(n)+\nu_{10}(n),
$$
and then the result follows from \eqref{P n+1} by using $\nu_{10}\big(d(x+y)+cw\big)=1$. 
\end{proof}

For a generic polynomial $P_0$ of degree 1 (i..e., a polynomial not passing through any of the points $p_1,\ldots,p_4$, for instance, $P_0=x$), we compute the following:
\medskip

\begin{center}
\begin{tabular}[h]{c|c|cccc}
\hline
$n$ & $d(n)$ & $\nu_3(n)$ & $\nu_6(n)$ & $\nu_9(n)$ & $\nu_{10}(n)$ \\
\hline
0 & 1 & 0 & 0 & 0 & 0\\
1 & 2 & 0 & 0 & 2 & 1 \\
2& 4 & 0 & 2 & 5 & 3 \\
3 & 8 & 2 & 5 & 11 & 7 \\
4 & 14 & 5 & 11 & 19 & 13\\
5 & 23 & 11 & 19 & 31 & 22 \\
6 & 35 & 19 & 31 & 46 & 34 \\
7 & 51 & 31 & 46 & 66 & 50\\
8 & 71 & 46 & 66 & 90 & 70\\
9 & 96 & 66 & 90 & 120  & 95 \\
10 & 126 & 90 & 120 & 155 & 125
\end{tabular}
\end{center}
\medskip

\begin{theo}
For any polynomial $P_0$, the degrees $d(n)=\deg(P_n)$ satisfy the following recurrence relation:
\begin{equation}\label{d recur}
    d(n+1)-3d(n)+2d(n-1)+2d(n-2)-3d(n-3)+d(n-4)=0.
\end{equation}
\end{theo}
\begin{proof}
Comparing \eqref{d n+1} with \eqref{mu4 2}, we see that
\begin{equation}\label{mu4 = d}
d(n+1)-\nu_{10}(n+1)=d(n)-\nu_{10}(n)=:\alpha.
\end{equation}
This allows us to re-write \eqref{mu3 2} as 
\begin{equation}\label{mu3 3}
\nu_9(n+1)=3 d(n)-2\nu_3(n)-\alpha, \quad n\ge 1.
\end{equation}
From \eqref{mu1 2}--\eqref{mu4 2}, \eqref{mu3 3} and \eqref{d n+1} we derive:
\begin{eqnarray*}
\nu_3(n)& = &  \nu_6(n-1)= \nu_9(n-2)\\
 & = & 3d(n-3)-2\nu_3(n-3)-\alpha\\
 & = & 2\big(2d(n-3)-\nu_3(n-3)\big)-d(n-3)-\alpha\\
 & = & 2d(n-2)-d(n-3)-\alpha.
\end{eqnarray*}
Plugging this into \eqref{d n+1} gives 
\begin{equation}\label{d recur alpha}
    d(n+1)=2d(n)-2d(n-2)+d(n-3)+\alpha.
\end{equation}
Subtracting from \eqref{d recur alpha} its shifted version,
\begin{equation}\label{d recur alpha shifted}
    d(n)=2d(n-1)-2d(n-3)+d(n-4)+\alpha,
\end{equation}
we finally arrive at \eqref{d recur}.
\end{proof}
Note that in the above table we have $\alpha=1$.

\section{Conclusion}
In this paper, we illustrated the technique for computing degrees of iterates of birational maps of $\mbP^N$, introduced and applied for $N=2$ in our previous paper \cite{Alonso_Suris_Wei_2D} to the case $N=3$. We illustrate our technique with three examples. In the simplest case, illustrated in Example 1 by a generalized discrete time Euler top, our method coincides with the method using the Picard group of the blow-up variety. In Example 2, dealing with a 3D representation of the discrete Painlev\'e I equation, our method deviates substantially from the Picard group method, since it produces a lift which is not algebraically stable. We mention, however, that an algebraically stable lift still exists, and the method of the Picard group is still applicable and leads to the same result. Example 3, an inflated QRT map, illustrates the situation where infinitely many blow-ups would be necessary to regularize the contraction of the critical set, but our method nevertheless succeeds in producing a finite system of recurrent relations for determining the degrees of the iterated map. We hope that our methods and worked examples will contribute towards a progress in general understanding of the dynamics of birational maps in higher dimensions.
\medskip

\textbf{Acknowledgement.}
This research is supported by the DFG Collaborative Research Center TRR 109 ``Discretization in
Geometry and Dynamics''.

\bibliographystyle{siam}

\end{document}